\documentclass[english]{amsart}

\usepackage{babel}
\usepackage{amsmath}
\usepackage{amsthm}
\usepackage{amssymb}
\usepackage[utf8]{inputenc}
\usepackage{graphicx}
\usepackage[colorlinks=true, linkcolor=blue, citecolor=blue]{hyperref}
\usepackage{tikz-cd}
\usepackage{dcolumn}
\usepackage{xcolor}
\usepackage{tikz}
\usepackage{mathtools}
\usepackage{nicematrix}
\usepackage{mathrsfs}
\usepackage{booktabs}
\usepackage{bbm}
\usepackage{enumitem}
\usepackage{scalerel}
\usetikzlibrary{matrix,arrows,chains}
\newtheorem*{theorem*}{Theorem}
\newtheorem{prop}{Proposition}[section]
\newtheorem{thm}[prop]{Theorem}
\newtheorem{cor}[prop]{Corollary}
\newtheorem{lem}[prop]{Lemma}

\theoremstyle{definition}
\newtheorem{rem}[prop]{Remark}
\newtheorem{de}[prop]{Definition}
\newtheorem{ex}[prop]{Example}

\DeclareMathOperator{\Set}{Set}

\DeclareMathOperator{\id}{id}

\DeclareMathOperator{\Hom}{Hom}

\newcommand{\iso}{\mathrel{\overset{\cong}{\smash{\longrightarrow}\vrule height.3ex
width0ex\relax}}}

\renewcommand{\ker}{\kerr}
\DeclareMathOperator{\kerr}{ker}

\renewcommand{\lim}{\limi}
\DeclareMathOperator{\limi}{lim}

\DeclareMathOperator{\Ab}{Ab}

\DeclareMathOperator{\Fld}{Fld}

\DeclareMathOperator{\GL}{GL}

\DeclareMathOperator{\rk}{rk}
\DeclareMathOperator{\Frac}{Frac}

\DeclareMathOperator{\fg}{fg}

\DeclareMathOperator{\SH}{SH}

\DeclareMathOperator{\Sm}{Sm}

\DeclareMathOperator{\Op}{Op}
\DeclareMathOperator{\Spec}{Spec}

\DeclareMathOperator{\HHom}{\underline{\Hom}}
\DeclareMathOperator{\codim}{codim}

\DeclareMathOperator{\G}{\mathbb{G}_m}

\DeclareMathOperator{\K}{K}
\DeclareMathOperator{\KK}{\underline{K}}

\DeclareMathOperator{\II}{\underline{I}}

\DeclareMathOperator{\GGW}{\underline{GW}}
\DeclareMathOperator{\WW}{\underline{W}}
\DeclareMathOperator{\Pf}{Pf}

\DeclareMathOperator{\MM}{M}
\DeclareMathOperator{\W}{W}
\DeclareMathOperator{\MW}{MW}
\DeclareMathOperator{\Q}{Q}
\DeclareMathOperator{\RR}{R}
\DeclareMathOperator{\HH}{H}

\DeclareMathOperator{\Z}{\underline{\mathbb{Z}}}

\DeclareMathOperator{\Gn}{\mathbb{G}^{\wedge \textit{n}}_m}
\DeclareMathOperator{\Gm}{\mathbb{G}}

\renewcommand{\sp}{\spp}
\DeclareMathOperator{\spp}{sp}

\DeclareMathOperator{\st}{st}
\DeclareMathOperator{\ttt}{tt}
\renewcommand{\tt}{\ttt}

\DeclareMathOperator{\mwkx}{\mathrel{\overset{\oplus \partial_{\nu_y}^{\pi_y}}{\smash{\longrightarrow}\vrule height.9ex
width0ex\relax}}}
\DeclareMathOperator{\MGL}{MGL}
\DeclareMathOperator{\MSL}{MSL}
\DeclareMathOperator{\KGL}{KGL}
\DeclareMathOperator{\KQ}{KQ}
\DeclareMathOperator{\BO}{BO}
\DeclareMathOperator{\BGL}{BGL}

\graphicspath{ {./Bilder/} }
\tikzset{%
    symbol/.style={%
        draw=none,
        every to/.append style={%
            edge node={node [sloped, allow upside down, auto=false]{$#1$}}}
    }
}

\begin{document}

\title{Operations on Milnor-Witt K-theory}
\author{Thor Wittich}
\address{\hspace*{1.5em}Mathematisches Institut\\
Heinrich-Heine-Universit\"at D\"usseldorf\\
40204 D\"us\-sel\-dorf\\
Germany}
\email{thor.wittich@hhu.de}

\begin{abstract}
For all positive integers $n$ and all homotopy modules $M_*$, we define certain operations $\KK^{\MW}_n \rightarrow M_*$ and show that these generate the $M_*(k)$-module of all (in general non-additive) operations $\KK^{\MW}_n \rightarrow M_*$ in a suitable sense, if $M_*$ is $\mathbb{N}$-graded and has a ring structure. This also allows us to explicitly compute the abelian group $\Op(\KK_n^{\MW},\KK_m^{\MW})$ and all operations between related theories such as Milnor, Witt and Milnor-Witt K-theory.
\end{abstract}
\maketitle

{ \hypersetup{hidelinks} \tableofcontents }

\section{Introduction}
Operations between invariants such as (co-)homology theories have shown to be very useful, especially in algebraic topology and algebraic geometry. Famous examples include Adams' treatment of vector fields on spheres \cite{MR139178} and Adams' and Atiyah's proof of the Hopf invariant one problem \cite{MR198460}, both using the so-called Adams' operations $\psi^l \colon K \rightarrow K$ on topological K-theory. Also stable operations, i.e. families of operations on a (co-)homology theory which respect the suspension functor are of major interest. Such operations on mod $p$ motivic cohomology $H^{*,*}(-;\mathbb{Z}/p\mathbb{Z})$ are for example constructed and used by Voevodsky in his proof of the Bloch-Kato conjecture \cite{MR2031199},\cite{MR2031198} and \cite{MR2811603}, nowadays known as the norm residue isomorphism theorem. If $F$ is a field, the motivic cohomology group $H^{n,n}(\Spec(F);\mathbb{Z}/p\mathbb{Z})$ is given by mod $p$ Milnor K-theory $K^{\MM}_n(F) / p$ and by integral Milnor K-theory $K^{\MM}_n(F)$ in the case of integral coefficients. For these theories defined over field extensions of a fixed base field $k$, Vial \cite{MR2497581} determines the $M_*(k) /p$-module of all uniformly bounded operations $K^{\MM}_n / p \rightarrow M_*$ and the $M_*(k)$-module of uniformly bounded operations $K^{\MM}_n \rightarrow M_*$, where $M_*$ is a cycle module with ring structure such as Milnor K-theory or algebraic K-theory. In both cases these modules turn out to be spanned by divided power operations. Actually, Vial also lifts the respective computations to the case of operations defined for smooth schemes over $k$. Note that the explicit assumption of uniform boundedness was forgotten by the author, as also Garrel states in Section 10 of \cite{MR4113769} with respect to the $p = 2$ case. 
\\Let $k$ be a perfect base field of characteristic not $2$. In motivic homotopy theory, there is a bigraded family of spheres $S^{n,m} = (S^1)^{\wedge (n-m)} \wedge \Gm_m^{\wedge m}$, where $n$ is the topological degree and $m$ is the algebraic degree or weight. These spheres give rise to the motivic sphere spectrum $\mathbbm{1}_k$ whose homotopy sheaves $\underline{\pi}_{n,m}(\mathbbm{1}_{k})$ are of central interest to motivic, but also to classical stable homotopy theory. The latter being the case due to work of Levine \cite{MR3217623}, which in particular yields isomorphisms $\pi_n(\mathbbm{1}) \cong \underline{\pi}_{n,0}(\mathbbm{1}_{\mathbb{C}})(\Spec(\mathbb{C}))$, where $\mathbbm{1}$ denotes the usual sphere spectrum from stable homotopy theory. In the other direction, Morel showed in \cite{MR2061856} that for all integers $m$, the sheaf $\underline{\pi}_{-m,-m}(\mathbbm{1}_k)$ has a purely algebraic description, called Milnor-Witt K-theory $\KK^{\MW}_m$ in degree $m$. By further work of Morel in \cite{MR2099118}, it arises as a pullback  square
\begin{center}
\begin{tikzcd}
 \KK^{\MW}_m \arrow[r] \arrow[d] & \KK^{\MM}_m \arrow[d]\\
 \II^m \arrow[r] & \KK^{\MM}_m / 2
\end{tikzcd}
\end{center}
of Nisnevich sheaves on $\Sm \! / k$, where $\II$ is the fundamental ideal of the Witt ring of quadratic forms. This allows us to view Milnor-Witt K-theory as a quadratic refinement of Milnor K-theory and also explains the name Milnor-Witt K-theory based on the fact that $\II^m$ also has an algebraic description called Witt K-theory, see \cite{MR2099118}. This leads to the natural question whether Vial's aforementioned results can be generalized to Milnor-Witt K-theory. The main goal of this paper is to explore and prove such generalizations. Our main strategy is to make tools of Garrel from the theory of quadratic forms available for Milnor-Witt K-theory, which we will describe now.

In \cite{MR4113769}, Garrel computes the modules of all operations $I^n \rightarrow W$ and $I^n \rightarrow H^*(-;\mu_2)$ defined over field extensions of a fixed base field $k$ of characteristic not $2$. Here also the natural filtrations of $W$ and $H^*(-;\mu_2)$ are respected, which therefore in particular gives all operations $I^n \rightarrow I^m$. Garrel's strategy relies on Theorem 18.1 of Serre \cite{MR1999384}, which describes all operations $\Pf_n \rightarrow W$ and $\Pf^n \rightarrow  H^*(-;\mu_2)$ as free modules of rank $2$. Here $\Pf_n$ are isomorphism classes of $n$-Pfister forms, i.e., the generators of $I^n$. Garrel then defines so-called shifts, which measure how an operation on $I^n$ changes when adding and subtracting generators $x \in \Pf_n$. In other words, the idea is to start with operations on generators and extend these to the entire theory, even though the operations must not be additive. This method is not available for Milnor-Witt K-theory, since operations on generators of Milnor-Witt K-theory are not known. This subsheaf of generators will be denoted by $[-_1,\dotsc,-_n]$ and our first main result clarifies operations on $[-_1,\dotsc,-_n]$.
\begin{theorem*}[see Theorem \ref{Operations on (K_1^M)^r that factorize through the multiplication map}]
For any homotopy module $M_*$ with ring structure and any positive integer $n$, the $M_*(k)$-module of operations $[-_1,\dotsc,-_n] \rightarrow M_*$ is free of rank $2$ generated by the constant operation $1$ and the action of $[-_1,\dotsc,-_n]$ on $1 \in M_*(k)$.
\end{theorem*} 
This allows us to define Garrel's shifts for operations on Milnor-Witt K-theory, but they are not as nicely behaved as in the case of quadratic forms, which mainly comes from the fact that Milnor-Witt K-theory is not commutative. 

In the case of quadratic forms, quite some operations are known, namely the pre-$\lambda$ operations of McGarraghy \cite{MR1901274} which come from exterior powers of quadratic forms. As shown by Garrel \cite{MR4113769}, these turn out to generate all operations in a suitable sense. So both for Milnor K-theory and for quadratic forms there are certain operations which are defined quite similarly and end up generating all other operations. There is also the recent preprint \cite{2209.08634} of Totaro in which a comparison of divided and exterior powers of quadratic forms is given. Additionally, it is shown, that these divided power operations are in fact compatible with the ones of Vial \cite{MR2497581} on Milnor K-theory mod $2$.
All in all, the natural guess is of course that also the operations on Milnor-Witt K-theory are spanned by such kinds of operations. This leads us to define operations $\lambda^n_l \cdot y \colon K^{\MW}_n \rightarrow M_*$  for every homotopy module $M_*$, which map sums of symbols
$$x = [a_{1,1},\dotsc,a_{1,n}] + \dotsc + [a_{r,1},\dotsc,a_{r,n}]$$
to the action of  
$$\sum_{1 \leq i_1 < \dotsc < i_l \leq r} [a_{i_1,1},\dotsc,a_{i_1,n}] \cdot \dotsc \cdot [a_{i_l,1},\dotsc,a_{i_l,n}]$$
on an element $y \in M_*(k)$. Such operations are not always well-defined. If $n$ is odd, the element $y$ needs to be $h$-torsion, which we denote by $y \in \prescript{}{h} M_*(k)$ (see Proposition \ref{Definition of Lambda} and Definition \ref{Definition of lamda}). Now that we also have some natural candidates for generators of all operations, we continue translating Garrel's strategy to Milnor-Witt K-theory. For this to work we restrict to $\mathbb{N}$-graded homotopy modules with ring structure $M_*$, which we call $\mathbb{N}$-graded homotopy algebras. Furthermore we consider certain linear combinations of operations of the form $\lambda^n_d$, namely
$$\sigma_l^n = \sum^{\lfloor \frac{l-1}{2} \rfloor}_{j = 0} \binom{\lfloor \frac{l-1}{2} \rfloor }{j}[-1]^{n(l-j)}\lambda_{l-j}^n$$
for all integers $l \geq 1$ and $\sigma_0^n = \lambda_0^n$
(see before Proposition \ref{Derivatives of new Lambdas}). Our next main result is the following.
\begin{theorem*}[see Theorem \ref{Operations K_n^MW -> M_*}]
Let $n$ be a positive integer. For all $\mathbb{N}$-graded homotopy algebras $M_*$, the map
$$f \colon M_*(k)^2 \times \prescript{}{\delta_n h} M_*(k)^{\mathbb{N} \setminus \lbrace 0,1 \rbrace} \rightarrow \Hom(\KK^{\MW}_n, M_*), \, (a_l)_{l \geq 0} \mapsto \sum_{l \geq 0} \sigma_l^n \cdot a_l$$
is an isomorphism of $M_*(k)$-modules, where $\delta_n $ is $1$ if $n$ is odd and $0$ if $n$ is even.
\end{theorem*}
Actually, we also have an appropriate filtration on both sides, which this isomorphism also respects. Here the right hand side is equipped with the natural filtration induced by the filtration $F_dM_* = M_{\geq d}$. It is not completely obvious that the map $f$ is well-defined, i.e. that infinite sums of the form $\sum_{l \geq 0} \sigma_l^n \cdot a_l$ with suitable coefficients $(a_l)_{l \geq 0}$ make sense. For this we show that while these sums are infinite, they become finite when evaluating at any element of $\KK^{\MW}_n$.

We then also recover and generalize the aforementioned results of Garrel and Vial, which leads to the following result.
\begin{theorem*}[see Theorem \ref{Operations between homotopy modules}]
Let $n$ be a positive integer. Table 1 gives a complete list of operations of type $(n,m)$ between Milnor, Witt and Milnor-Witt K-theory.
\end{theorem*}
\begin{center}
Table 1
\end{center}
\begin{center}
 \def\arraystretch{1.5}
  \begin{tabular}{ p{0.8cm} p{0.8cm} p{10.17cm}@{}  }
     $A$ & $B$ & $\Hom(A,B)$ \\ \midrule
   $\KK^{\MM}_n$ & $\KK^{\MM}_m$ & $\displaystyle{\Bigl\{ \sum_{l \geq 0} \overline{\sigma}_l^n \cdot a_l \Bigm\lvert (a_l)_l \in \hspace*{-9pt}\prod_{\min(\frac{m}{n},1) \geq l \geq 0} \hspace*{-9pt}K^{\MM}_{m-nl}(k) \times \hspace*{-5pt}\prod_{\frac{m}{n}\hspace*{-1pt} \geq l \geq 2} \prescript{}{\delta_n 2}{ (\prescript{}{\tau_n} K_{m-nl}^{\MM}(k) )}  \Bigr\}}$ \\
   $\KK^{\MM}_n$ & $\KK^{\W}_m$ & $\displaystyle{\Bigl\{ \sum_{l \geq 0} \overline{\sigma}_l^n \cdot a_l \Bigm\lvert (a_l)_l \in \prod_{l = 0}^1 K^{\W}_{m-nl}(k) \times \prod_{l \geq 2} \,\prescript{}{\delta_n 2}{ (\prescript{}{\tau_n} K_{m-nl}^{\W}(k) )}  \Bigr\}}$ \\ 
   $\KK^{\MM}_n$ & $\KK^{\MW}_m$ & $\displaystyle{\Bigl\{ \sum_{l \geq 0} \overline{\sigma}_l^n \cdot a_l \Bigm\lvert (a_l)_l \in \prod_{l = 0}^1 K^{\MW}_{m-nl}(k) \times \prod_{l \geq 2} \,\prescript{}{\delta_n 2}{ (\prescript{}{\tau_n} K_{m-nl}^{\MW}(k) )}  \Bigr\}}$ \\
   $\KK^{\W}_n$ & $\KK^{\MM}_m$ &$\displaystyle{\Bigl\{ \sum_{l \geq 0} \overline{\sigma}_l^n \cdot a_l \Bigm\lvert (a_l)_l \in K^{\MM}_m(k) \times \prod_{\frac{m}{n}\hspace*{-1pt} \geq l \geq 1} \,\prescript{}{2}K^{\MM}_{m-nl}(k) \Bigr\}}$ \\
   $\KK^{\W}_n$ & $\KK^{\W}_m$ & $\displaystyle{\Bigl\{ \sum_{l \geq 0} \overline{\sigma}_l^n \cdot a_l \Bigm\lvert (a_l)_l \in \prod_{l \geq 0} K^{\W}_{m-nl}(k) \Bigr\}}$ \\ 
   $\KK^{\W}_n$ & $\KK^{\MW}_m$ & $\displaystyle{\Bigl\{ \sum_{l \geq 0} \overline{\sigma}_l^n \cdot a_l \Bigm\lvert (a_l)_l \in K^{\MW}_m(k) \times \prod_{l \geq 1} \,\prescript{}{h}K^{\MW}_{m-nl}(k) \Bigr\}}$ \\
   $\KK^{\MW}_n$ & $\KK^{\MM}_m$ & $\displaystyle{\Bigl\{ \sum_{l \geq 0} \sigma_l^n \cdot a_l \Bigm\lvert (a_l)_l \in \prod_{\min(\frac{m}{n},1) \geq l \geq 0} K^{\MM}_{m-nl}(k) \times \hspace*{-5pt}\prod_{\frac{m}{n}\hspace*{-1pt} \geq l \geq 2} \,\prescript{}{\delta_n 2}K^{\MM}_{m-nl}(k) \Bigr\}}$ \\
   $\KK^{\MW}_n$ & $\KK^{\W}_m$ & $\displaystyle{\Bigl\{ \sum_{l \geq 0} \sigma_l^n \cdot a_l \Bigm\lvert (a_l)_l \in \prod_{l \geq 0} K^{\W}_{m-nl}(k) \Bigr\}}$ \\ 
   $\KK^{\MW}_n$ & $\KK^{\MW}_m$ & $\displaystyle{\Bigl\{ \sum_{l \geq 0} \sigma_l^n \cdot a_l \Bigm\lvert (a_l)_l \in \prod_{l = 0}^1 K^{\MW}_{m-nl}(k) \times \prod_{l \geq 2} \,\prescript{}{\delta_n h}K^{\MW}_{m-nl}(k) \Bigr\}}$ \\
  \end{tabular}
\end{center}
Here $\tau_n$ is the action of $[-1]^{n-1}$ on the respective homotopy module. 

The paper is organized as follows: In Section \ref{Section 2} we state our  conventions and give a list of some of our notations. Section \ref{Section 3} is an introduction to Milnor-Witt K-theory. We start by defining it for fields and list all relations that we consider essential. Afterwards we discuss generators and recall how to lift Milnor-Witt K-theory from fields to smooth schemes. Here we also explain how Milnor-Witt K-theory arises as a pullback of other K-theories. In Section \ref{Section 4}, we give a quick recollection of homotopy modules and some of their properties. Now that the preliminaries are done, we quickly compute the additive and $\Gm_m$-stable operations on Milnor-Witt K-theory in Section \ref{Section 5}, using the fact that Milnor-Witt K-theory in degree $n \geq 1$ is the free strictly $\mathbb{A}^1$-invariant sheaf on $\Gm_m^{\wedge n}$. In Section \ref{Section 6} we define those operations which in a suitable sense turn out to generate all operations on Milnor-Witt K-theory. In Section \ref{Section 7} we compute all operations on generators of Milnor-Witt K-theory, which allows us to define Garrel's shifts in our setting. Afterwards we start investigating their properties. Section \ref{Section 8} deals with the computation of all operations on Milnor-Witt K-theory based on the results from the previous sections. Finally, we recover the aforementioned results of Garrel and Vial in Section \ref{Section 9}.

\subsection*{Acknowledgement} This research was conducted in the framework of the research training group 2240: Algebro-geometric Methods in Algebra, Arithmetic and Topology. I would like to thank my supervisor Immanuel Halup\-czok for having interest in my work even though it is not related to his own research interests and, of course, for his help. I also wish to thank Marcus Zibrowius for bringing the problem of computing operations on Milnor-Witt K-theory to our attention. Finally, I thank Marcus Zibrowius, Matthias Wendt and Jan Hennig for fruitful discussions at various stages of this project and Holger Kammeyer for some useful comments.

\section{Conventions and Notations}\label{Section 2}
Throughout the article we let $k$ be a perfect base field of characteristic not $2$. Furthermore we assume all schemes to be separated and of finite type over $k$ and rings are not necessarily commutative. For the convenience of the reader, we give the following table of notations:
\begin{center}
\begin{tabular}{ l l l }
$\Fld_k$ & & Category of field extensions of $k$\\
$\Fld^{\fg}_k$ & & Category of finitely generated field extensions of $k$\\
$K_*^{\MM}$ & & Milnor K-theory as a functor on $\Fld \! / k$ or $\Fld^{\fg} \! / k$\\
$K_*^{\W}$ & & Witt K-theory as a functor on $\Fld \! / k$ or $\Fld^{\fg} \! / k$\\
$K_*^{\MW}$ & & Milnor-Witt K-theory as a functor on $\Fld \! / k$ or $\Fld^{\fg} \! / k$\\
$GW$ & & Grothendieck-Witt ring as a functor on $\Fld \! / k$ or $\Fld^{\fg} \! / k$\\
$W$ & & Witt ring as a functor on $\Fld \! / k$ or $\Fld^{\fg} \! / k$\\
$I$ & & Fundamental Ideal as a functor on $\Fld \! / k$ or $\Fld^{\fg} \! / k$\\
$\Sm \! / k$ & & Category of smooth schemes over $k$ with the Nisnevich topology\\
$\Set \! / k$ & & Category of sheaves (of sets) on $\Sm \! /k$\\ 
$\Set_* \! / k$ & & Category of sheaves of pointed sets on $\Sm \! /k$\\
$\Ab \! / k$ & & Category of abelian sheaves on $\Sm \! /k$\\ 
$\Ab_{\mathbb{A}^1} \! / k$ & & Category of strictly $\mathbb{A}^1$-invariant sheaves on $\Sm \! /k$\\
$\Z_{\mathbb{A}^1}[X]$ & & Free strictly $\mathbb{A}^1$- invariant sheaf on $X$\\
$\tilde{\Z}_{\mathbb{A}^1}[X]$ & & Reduced free strictly $\mathbb{A}^1$- invariant sheaf on $X$\\
$\SH(k)$ & & Motivic stable homotopy category over $k$\\
$\Pi_*(k)$ & & Category of homotopy modules over $k$\\
$\KK_*^{\MM}$ & & Milnor K-theory as a homotopy module\\
$\KK_*^{\W}$ & & Witt K-theory as a homotopy module\\
$\KK_*^{\MW}$ & & Milnor-Witt K-theory as a homotopy module\\
$\GGW$ & & Grothendieck-Witt ring as an unramified sheaf on $\Sm \! / k$\\
$\WW$ & & Witt ring as an unramified sheaf on $\Sm \! / k$\\
$\II$ & & Fundamental Ideal as an unramified sheaf on $\Sm \! / k$\\
$\prescript{}{x} M_*$ & & $x$-torsion of some homotopy module $M_*$\\
$\delta_n$ & & $1$ if $n$ is odd and $0$ if $n$ is even\\
$\Op_{\sp}$ & & Operations on field extensions commuting with specialization maps \\
$\tau_n$ & & action of $[-1]^{n-1}$ on some homotopy module \\

\end{tabular}
\end{center}
\section{Recollection of Milnor-Witt K-theory}\label{Section 3}
Let us quickly recall some basics of Milnor-Witt K-theory. For more details we refer to Chapters 3.1 and 3.2 of \cite{MR2934577}.
\begin{de}
The Milnor-Witt K-theory ring $K^{\MW}_*(F)$ of a field $F$ is the free unital $\mathbb{Z}$-graded ring generated by symbols $[a]$ of degree $1$ for all $a \in F^\times$ and a symbol $\eta$ of degree $-1$ subject to the following relations:
\begin{enumerate}[leftmargin=1.89cm]
\item[(MW1)] $[a][1-a] = 0$ for all $a \in F^\times \setminus \lbrace 1 \rbrace$ \hspace*{\fill} (Steinberg relation);
\item[(MW2)] $[ab] = [a] + [b] + \eta[a][b]$ for all $a,b \in F^\times$\hspace*{\fill} (Twisted tensor relation);
\item[(MW3)] $\eta [a] = [a] \eta$ for all $a \in F^\times$;
\item[(MW4)] $\eta(2 + \eta[-1]) = 0$\hspace*{\fill} (Witt relation).
\end{enumerate}
\end{de}
So an arbitrary element of degree $n$ in $K^{\MW}_*(F)$ is a sum of elements of the form $\eta^d[a_1,\dotsc,a_r]$ with $r - d = n$,  where we denote the product $[a_1] \cdot \dotsc \cdot [a_r]$ by $[a_1,\dotsc,a_r]$.  Furthermore we set $\langle a \rangle = 1 + \eta[a] \in K_0^{\MW}(F)$ for all $a \in F^\times$, $h = \langle 1 \rangle + \langle -1 \rangle$ and $\epsilon = -\langle -1 \rangle$. A summary of essential relations is:
\begin{lem}\label{List of relations of Milnor-Witt K-theory}
We have the following:
\begin{enumerate}
\item[(i)] $0 = [1]$ and $1 = \langle 1 \rangle$ as elements of $K^{\MW}_*(F)$. In particular, $h = 2 + \eta[-1]$ and the fourth defining relation of $K^{\MW}_*(F)$ can be rewritten as $\eta h = 0$.
\item[(ii)] $xx'$ = $\epsilon^{nm}x'x$ for all homogeneous elements $x,x' \in K^{\MW}_*(F)$ of degrees $n$ and $m$ respectively, i.e. the ring $K^{\MW}_*(F)$ is $\epsilon$-graded commutative.
\item[(iii)] $[a,-a] = 0 = [-a,a]$ for all $a \in F^\times$.
\item[(iv)] $[a,-1] = [a,a] = [-1,a]$ for all $a \in F^\times$. In particular $\langle a \rangle [a] = \langle -1 \rangle [a]$ for all $a \in F^\times$.
\item[(v)] $[a^n] = \sum_{i = 0}^{n-1} \langle (-1)^i \rangle [a]$ for all positive $n$ and $[a^n] = \epsilon \sum_{i = 0}^{-(n+1)} \langle (-1)^i \rangle [a]$ for all negative $n$ and all $a \in F^\times$. In particular, $[a^2] = h[a]$ for all $a \in F^\times$.
\item[(vi)] $\langle a \rangle \langle b \rangle = \langle ab \rangle$ for all $a,b \in F^\times$. Together with (i) this in particular yields that $\langle a \rangle$ is a unit with inverse $\langle a^{-1} \rangle$ and that $\epsilon^2 = 1$.
\item[(vii)] $\langle a \rangle^2 = \langle a^2 \rangle = 1$ for all $a \in F^\times$. 
\item[(viii)] $\langle a \rangle x = x \langle a \rangle$ for all $x \in K_*^{\MW}(F)$ and all $a \in F^\times$.
\item[(ix)] $\langle a \rangle[b] = [ab] - [a]$ for all $a,b \in F^\times$. 
\end{enumerate}
\end{lem}
All these relations will be used freely in all of our computations. Therefore we certainly want to encourage the reader to check this list of relations in case that some computation is unclear.
\begin{lem}\label{Generators of positive Milnor-Witt K-theory groups}
For all $n \geq 1$, the abelian group $K_n^{\MW}(F)$ is generated by elements of the form $[a_1,\dotsc,a_n]$ with $a_1,\dotsc,a_n \in F^\times$ and for all $n \leq 0$, the abelian group $K_n^{\MW}(F)$ is generated by elements of the form $\eta^n \langle a \rangle$ with $a \in F^\times$.
\end{lem}
We will mostly make use of this statement in the case $n \geq 1$. Here the proof merely consists of getting rid of powers of $\eta$ in elements of the form $\eta^d[a_1,\dotsc,a_{n+d}]$ by using relation (MW2) often enough, thus resulting in the pure ($\eta$-free) symbols as generators. Let us note that a list of relations with respect to these generators was computed by Hutchinson-Tao for $n \geq 2$ in \cite{MR2422514} and by Tao/Hutchinson-Tao for $n = 1$ and in \cite{MR2995531} and \cite{MR3078662}. We prefer to use the following standard presentation:
\begin{lem}\label{Standard presentation of positive degree Milnor-Witt K-theory}
For $n \geq 1$, the $n$-th Milnor Witt K-theory group $K_n^{\MW}(F)$ of $F$ is generated by elements of the form $\eta^d[a_1,\dotsc,a_r]$ with $d = r-n \geq 0$ subject to the relations:
\begin{enumerate}
\item[(i)] $\eta^d[a_1,\dotsc,a_r] = 0$ whenever $a_i + a_{i+1} = 1$ for some $1 \leq i \leq r-1$.
\item[(ii)] $\eta^d[a_1,\dotsc,a_{i-1},bb',a_{i+1},\dotsc,a_r] = \eta^d[a_1,\dotsc,a_{i-1},b,a_{i+1},\dotsc,a_r] \\ \hspace*{138pt} + \eta^d[a_1,\dotsc,a_{i-1},b',a_{i+1},\dotsc,a_r] \\ \hspace*{138pt} + \eta^{d+1}[a_1,\dotsc,a_{i-1},b,b',a_{i+1},\dotsc,a_r]$ 
\item[(iii)] $2\eta^{d+1}[a_1,.\hspace*{1pt}.\hspace*{1pt}.\hspace*{1pt},a_{i-1},a_{i+1},.\hspace*{1pt}.\hspace*{1pt}.\hspace*{1pt},a_{r+2}] + \eta^{d+2}[a_1,.\hspace*{1pt}.\hspace*{1pt}.\hspace*{1pt},a_{i-1},-1,a_{i+1},.\hspace*{1pt}.\hspace*{1pt}.\hspace*{1pt},a_{r+2}] = 0$
\end{enumerate}
\end{lem}
Milnor-Witt K-theory is functorial with respect to field extensions. This allows us to view both $K_*^{\MW}$ and $K_n^{\MW}$ for a fixed integer $n$ as $\Ab$- and $\Set$-valued functors on the categories $\Fld_k$ of field extensions and $\Fld^{\fg}_k$ of finitely generated field extensions of a our base field $k$. The more general definition of Milnor-Witt K-theory of smooth schemes requires certain maps. For this suppose that $F$ is a discretely valued field with valuation $\nu$ and fixed uniformizing element $\pi$. We denote the associated valuation ring by $\mathcal{O}_\nu$ and the residue field by $\kappa(\nu)$.
\begin{thm}
There is exactly one homomorphism $$\partial_\nu^{\pi} \colon K_*^{\MW}(F) \rightarrow K_{*-1}^{\MW}(\kappa(\nu))$$ of graded abelian groups with the following three properties:
\begin{enumerate}
\item[(i)] $\partial_\nu^{\pi}(\eta x) = \eta \partial_\nu^{\pi}(x)$ for all $x \in K_*^{\MW}(F)$.
\item[(ii)] $\partial_\nu^{\pi}([\pi,u_2,\dotsc,u_n]) = [\overline{u_2},\dotsc, \overline{u_n}]$ for all $u_2,\dotsc,u_n \in \mathcal{O}_\nu^\times$.
\item[(iii)] $\partial_\nu^{\pi}([u_1,u_2,\dotsc,u_n]) = 0$ for all $u_1,\dotsc,u_n \in \mathcal{O}_\nu^\times$.
\end{enumerate}
\end{thm}
This homomorphism is called residue map and the composition
$$s_\nu^{\pi} \colon K_*^{\MW}(F) \xrightarrow{[-\pi]\cdot} K_{*+1}^{\MW}(F) \xrightarrow{\partial_\nu^{\pi}} K_*^{\MW}(\kappa(\nu)) \xrightarrow{\langle \overline{-1} \rangle \cdot} K_*^{\MW}(\kappa(\nu))$$
is called specialization map and is a homomorphism of graded rings. Let $a \in F^\times$ and write $a = \pi^n u$ for some unit $u \in \mathcal{O}_{\nu}^\times$. Then the specialization map can also be defined as the unique homomorphism $K_*^{\MW}(F) \rightarrow K_*^{\MW}(\kappa(\nu))$ of graded rings mapping $[\pi^n u]$ to $[\overline{u}]$ and $\eta$ to $\eta$. Some useful relations of these two kinds of maps are:
\begin{prop}\label{Properties of residue und specialization maps}
For all $u \in \mathcal{O}_\nu^\times$ and all $x \in K_*^{\MW}(F)$ we have:
\begin{enumerate}
\item[(i)] $\partial_\nu^{\pi}([u]x) = \epsilon [\overline{u}]\partial_\nu^{\pi}(x)$ and $s_\nu^{\pi}([u]x) = [ \overline{u} ] s_\nu^{\pi}(x)$.
\item[(ii)] $\partial_\nu^{\pi}(\langle u \rangle x) = \langle \overline{u}  \rangle\partial_\nu^{\pi}(x)$ and $s_\nu^{\pi}(\langle u \rangle x) = \langle \overline{u}  \rangle s_\nu^{\pi}(x)$.
\item[(iii)] $\partial_\nu^{u\pi}(x) = \langle \overline{u}  \rangle\partial_\nu^{\pi}(x)$ and $s_\nu^{u\pi}(x) = s_\nu^{\pi}(x) + \epsilon[\overline{u}]\partial_\nu^{\pi}(x)$. In particular, both the residue map and the specialization map do generally depend on the choice of the uniformizing element $\pi$.
\end{enumerate}
\end{prop}
\begin{proof}
The first two relations for residue maps are Proposition 3.17 in \cite{MR2934577} and the corresponding ones for the specialization maps follow immediately from the definition and the respective relations for the residue map. This clarifies (i) and (ii). The first formula of (iii) is Remark 1.9 in \cite{MR4071213}. We will quickly prove the second one. Let $u \in \mathcal{O}_\nu^\times$ and $x \in K_*^{\MW}(F)$. Then
\begin{align*}s_\nu^{u\pi} = \langle \overline{-1} \rangle \partial^{u\pi}_{\nu}([-u\pi]x) = \langle \overline{-1} \rangle \langle \overline{u} \rangle \partial^{\pi}_{\nu}([-u\pi]x) &= \langle \overline{-1} \rangle \langle \overline{u} \rangle \partial^{\pi}_{\nu}((\langle u \rangle [-\pi] + [u])x) \\ & = \langle \overline{-1} \rangle( \partial^{\pi}_{\nu}([-\pi]x) + \epsilon \langle \overline{u} \rangle[ \overline{u}]\partial^{\pi}_{\nu}(x)) \\ & = \langle \overline{-1} \rangle \partial^{\pi}_{\nu}([-\pi]x) - \langle \overline{u} \rangle[ \overline{u}]\partial^{\pi}_{\nu}(x)
\\ & = s_\nu^{\pi}(x) + \epsilon[\overline{u}]\partial^{\pi}_{\nu}(x),
\end{align*}
where we use that the residue map is a group homomorphism which satisfies the two formulas from (i), (ii) and (iii).
\end{proof}
Recall that every closed point $p \in \mathbb{A}^1,$ or equivalently every monic irreducible polynomial $f \in F[t]$, gives rise to a discrete valuation on $F(t)$, which measures the divisibility with respect to $f$. We will denote this valuation by $v_p$ or $v_f$ and the associated residue map with respect to the uniformizer $f$ by $\partial^p_{v_p}$ or $\partial^f_{v_f}$. These residue maps allow us to express Milnor-Witt of K-theory of $F(t)$ in terms of Milnor-Witt of K-theory of $F$ and Milnor-Witt of K-theory of the residue fields $\kappa(v_p)$:
\begin{thm}
There is a split short exact sequence 
\begin{center}
\begin{tikzcd}
0 \arrow[r] & K_*^{\MW}(F) \arrow[r, "i_*"] & K_*^{\MW}(F(t))  \arrow[r, "\oplus_p \partial^p_{\nu_p}"] & \bigoplus\limits_{p \in \mathbb{A}^1} K_{*-1}^{\MW}(\kappa(v_p))  \arrow[r] & 0
\end{tikzcd}
\end{center}
of graded $K_*^{\MW}(F)$-modules, where $i_*$ is the map induced by the inclusion $F \subset F(t)$.
\end{thm}
A retraction of $i_*$ is given by the specialization map $s^t_{v_t}$. This kind of sequence is usually refered to as Milnor's short exact sequence due to Milnor's seminal paper \cite{MR260844}, where he constructs this type of sequence for both Milnor K-theory and Witt rings of quadratic forms.
\\For a scheme $X$ we denote by $X^{(c)}$ the set of its points of codimension $c$. Recall that one can define a discrete valuation ring to be a normal noetherian local domain of dimension $1$. Therefore, if we are given a smooth irreducible scheme $X,$ any point $x \in X^{(1)}$ gives rise to a discrete valuation $v_x$ on $\Frac(\mathcal{O}_{X,x}) = k(X)$. Indeed, the local ring is noetherian since $X$ is locally noetherian. It is normal since $X$ is smooth and its dimension is $\dim(\mathcal{O}_{X,x}) = \codim(x) = 1$. If $X$ is reducible, then the same holds for all codimension $1$ points $y \in X^{(1)}$ in the closure of a given codimension $0$ point $x \in X^{(0)}$. In particular, we get residue maps $\partial_{v_y}^{\pi_y} \colon K_*^{\MW}(k(x)) \rightarrow K_{*-1}^{\MW}(k(y))$ for any choice of uniformizing elements $\pi_y$.
\begin{de}\label{Milnor Witt K-theory of a scheme}
The $n$-th Milnor-Witt K-theory group of a smooth scheme $X$ is
$$K_n^{\MW}(X) = \ker \biggl( \bigoplus_{x \in X^{(0)}}K_n^{\MW}(k(x)) \mwkx \bigoplus_{y \in X^{(1)}}K_{n-1}^{\MW}(k(y))\biggr)$$
\end{de}
Note that this does not depend on the choices of uniformizers by Proposition \hyperref[Properties of residue und specialization maps]{\ref{Properties of residue und specialization maps}} and is hence well-defined. If we are given a morphism $f \colon X \rightarrow Y$ between smooth schemes, one can define a pullback map $f^* \colon K_n^{\MW}(Y) \rightarrow K_n^{\MW}(X)$ as follows:

As a morphism between smooth schemes, $f$ is a local complete intersection morphism and thus factorizes as a regular embedding $j \colon X \rightarrow Z$ followed by a smooth morphism $g \colon Z \rightarrow Y$. The idea is to  proceed by constructing the desired pullback maps for $j$ and $g$. For $g$ this is relatively straight-forward and is really just given by a pullback of symbols. For the morphism $j$, one reduces to the codimension $1$ case via a choice of a regular sequence and realizes the pullback by using specialization maps. Finally, there is work to be done to verify that everything is independent of all the choices that were made. The resulting $n$-th Milnor-Witt K-theory sheaf will be denoted by $\KK^{\MW}_n$, whereas we will still use the notation $K^{\MW}_n$ for its ``restriction" to the category $\Fld^{\fg}_k$. Such kinds of sheaves are called unramified and their construction is the main content of Chapter 2.1 of \cite{MR2934577}. 

Given two unramified sheaves $M$ and $N$, we can ``restrict" them to the category $\Fld^{\fg}_k$ of finitely generated field extension of $k$ and just consider those morphisms which commute with specialization maps. By this we mean morphisms $\varphi \colon \left.M\right|_{\Set^{\Fld^{\fg}_k}} \rightarrow \left.N\right|_{\Set^{\Fld^{\fg}_k}}$ making the diagram 
\begin{center}
\begin{tikzcd}
 \left.M \right|_{\Set^{\Fld^{\fg}_k}}(F) \arrow[r, "\varphi"] \arrow[d, "s_{\nu}^{\pi}"] & \left.N \right|_{\Set^{\Fld^{\fg}_k}}(F) \arrow[d, "s_{\nu}^{\pi}"]\\
 \left.M \right|_{\Set^{\Fld^{\fg}_k}}(\kappa(\nu)) \arrow[r, "\varphi"] & \left.N \right|_{\Set^{\Fld^{\fg}_k}}(\kappa(\nu)).
\end{tikzcd}
\end{center}
commutative for every finitely generated field extension $k \subset F$ equipped with a discrete valuation $\nu$, residue field $\kappa(\nu)$ containing $k$ and every choice of uniformizing element $\pi$. The set of such morphisms will be denoted by $\Op_{\sp}(M,N)$. The proof of the equivalence of categories between unramified sheaves and so-called unramified $\Fld^{\fg}_k$-data found as Theorem 2.11 in \cite{MR2934577} in particular yields:
\begin{prop}\label{Morphisms between Unramified Sheaves}
For all unramified sheaves $M$ and $N$ we have an identification $$\Hom_{\Set \! / k}(M, N) = \Op_{sp}(\left.M\right|_{\Set^{\Fld^{\fg}_k}},\left.N\right|_{\Set^{\Fld^{\fg}_k}}).$$
\end{prop}
Therefore we will from now on mostly restrict to the category $\Fld^{\fg}_k$ of finitely generated field extensions of our base field $k$.

Milnor-Witt K-theory also arises as a certain pullback. For this note that the quotient $\KK^{\MW}_* \! / \eta$ by definition is given by Milnor K-theory $\KK^{\MM}_*$. A second quotient we consider is $\KK^{\W}_* \cong \KK^{\MW}_* \! / h$, called Witt K-theory. It was defined by Morel in terms of generators and relations similar to Milnor-Witt K-theory \cite{MR2099118} and the isomorphism $\KK^{\W}_* \iso \KK^{\MW}_* \! / h$ is given by mapping $\eta$ to $\eta + h \! \KK^{\MW}_*$ and a symbol $\lbrace a \rbrace$ to the class $-[a] + h\! \KK^{\MW}_*$. Furthermore, Morel showed that Witt K-theory is nothing but the graded ring of powers of the fundamental ideal $\bigoplus_{n \in \mathbb{Z}} \II^n$, where by convention $\II^n = \WW$ for negative $n$. Here the isomorphism identifies pure symbols $\lbrace a_1,\dotsc,a_n\rbrace$ of length $n$ with $n$-Pfister forms $\langle \langle a_1,\dotsc,a_n \rangle\rangle$ and the multiplication with $\eta$ with the inclusions $\II^{n+1} \hookrightarrow \II^n$. The resolution of the Milnor conjecture of quadratic forms by Orlov-Vishik-Voevodsky \cite{MR2276765} therefore gives a diagram 
\begin{center}
\begin{tikzcd}
 \KK^{\MW}_* \arrow[rr] \arrow[dd] & & \KK^{\MM}_* \arrow[d]\\
  &  & \KK^{\MM}_* \! / 2\\
 \KK^{\W}_* \arrow[r] & \KK^{\W}_* \! /\eta \! \KK^{\W}_{*+1} \arrow[ur, "\cong"] & 
\end{tikzcd}
\end{center}
Morel \cite{MR2099118} showed that this is a cartesian square when applied to fields, which then by the content of Chapters 2 und 3 of \cite{MR2934577} extends to the case of sheaves, as all these maps are compatible with the respective residue/specialization maps. This in particular recovers the classical pullback square
\begin{center}
\begin{tikzcd}
 \GGW \arrow[r] \arrow[d] & \Z \arrow[d]\\
 \WW \arrow[r] & \underline{\mathbb{Z}/2\mathbb{Z}}
\end{tikzcd}
\end{center}
in degree $0$ and shows that $\KK_n^{\MW} \cong \WW$ for negative integers $n$. The pullback square for Milnor-Witt K-theory allows us to study operations $\KK^{\MW}_n \rightarrow \KK^{\MW}_m$ by studying operations $\KK^{\MW}_n \rightarrow \KK^{\RR}_m$, where the later can stand for Milnor K-theory, Milnor K-theory mod 2 or Witt K-theory.
\section{Recollection of Homotopy Modules}\label{Section 4}
Although we are mainly interested in operations $\KK^{\MW}_n \rightarrow \KK^{\MW}_m$ and hence in operations $\KK^{\MW}_n \rightarrow \KK^{\RR}_m$, quite some of our arguments work for more general targets, namely homotopy modules.
\begin{de}
A homotopy module $(M_*,\epsilon_*)$ is a sheaf $M_* \in (\Ab_{\mathbb{A}^1} \! /k)^\mathbb{Z}$ together with isomorphisms $\epsilon_n \colon M_n \rightarrow (M_{n+1})_{-1}$ for all $n \in \mathbb{Z}$. If it additionally has the structure of a graded ring when applied to fields, we call it a homotopy algebra.
\end{de}
Here $( \,\,\,\,)_{-1} \colon \Ab_{\mathbb{A}^1} \! /k \rightarrow \Ab_{\mathbb{A}^1} \! /k$ is the homotopy contraction, which can be defined as $(-)^{\tilde{\Z}_{\mathbb{A}^1}[\G]} = \HHom_{\Ab_{\mathbb{A}^1} \! /k}(\tilde{\Z}_{\mathbb{A}^1}[\G],-)$ where $\tilde{\Z}_{\mathbb{A}^1}[X]$ is the kernel of the map $\Z_{\mathbb{A}^1}[X] \rightarrow \Z_{\mathbb{A}^1}[\Spec(k)] = \Z$ for some point $\Spec(k) \rightarrow X$. This makes $\tilde{\Z}_{\mathbb{A}^1}$ the left adjoint to the forgetful functor $\Ab_{\mathbb{A}^1} \! / k \rightarrow \Set_*  \! / k$. We will usually drop the isomorphisms $\epsilon_*$ from the notation and just say that $M_*$ is a homotopy module. A morphisms of homotopy modules is a morphism of $\mathbb{Z}$-graded abelian sheaves $f \colon M_* \rightarrow N_*$ which respects the given isomorphisms. The category $\Pi_*(k)$ of homotopy modules arises naturally as follows. The restriction of the functor
$$\underline{\pi}_0(-)_* \colon \SH(k) \rightarrow \Pi_*(k), E \mapsto \underline{\pi}_0(E)_* = \bigoplus_{m \in \mathbb{Z}} \underline{\pi}_{-m,-m}(E)$$
to the heart of the homotopy $t$-structure $$\SH(k)^\heartsuit = \lbrace E \in \SH(k) \mid \underline{\pi}_n(E)_* = 0 \text{ for all } n \neq 0 \rbrace$$ defines an equivalence of categories. Here the sheaf $\underline{\pi}_0(E)_*$ is equipped with the canonical isomorphisms $\epsilon_* \colon \underline{\pi}_0(E)_* \rightarrow (\underline{\pi}_0(E)_{*+1})_{-1}$. As a consequence, the category $\Pi_*(k)$ of homotopy modules is an abelian category. We say that a homotopy module $M_*$ is associated to a motivic spectrum $E \in \SH(k)$, if $\underline{\pi}_0(E)_* \cong M_*$. Note that we do not demand that $E$ lies in $\SH(k)^\heartsuit$.
\begin{ex}
Milnor Witt K-theory $\KK_*^{\MW}$ is the homotopy module associated to the motivic sphere spectrum $\mathbb{S}_k$ as shown by Morel \cite[Theorem~6.2.1]{MR2061856}, the motivic spectrum $\tilde{H}\mathbb{Z}$ representing Milnor-Witt motivic cohomology, see e.g. Déglise and Fasel \cite[Theorem~4.2.3]{2004.06634} and the algebraic special linear cobordism spectrum $\MSL$ by work of Yakerson \cite[Proposition~3.6.3]{MR4332780}.
\end{ex}
In particular, the $\mathbb{S}_k$-module structure on any motivic spectrum $E$ gives rise to a $\KK_*^{\MW}$-module structure on the homotopy module $\underline{\pi}_0(E)_*$, so that every homotopy module is equipped with such structure. This can also be seen via Chapter 2.3 of \cite{MR2934577} or Feld's theory of Milnor-Witt cycle modules, see \cite{MR4282798}.
\begin{ex}
As quotients of Milnor-Witt K-theory, Milnor K-theory $\KK^{\MM}_*$, Witt K-theory $\KK^{\W}_*$ and $\KK^{\MM}_* \! / 2$ are homotopy modules. Here the $\KK^{\MW}_*$-actions are given via the quotient maps. Milnor K-theory also arises as the homotopy module associated to the motivic Eilenberg-Maclane spectrum $H\mathbb{Z}$ and the algebraic cobordism spectrum $\MGL$ and $\KK^{\MM}_* \! / 2$ as the homotopy module associated to $H\mathbb{Z}/2$, see Theorem 3.4 of \cite{MR1744945} and Remark 3.10 of \cite{MR3341470}.
\end{ex}
\begin{ex}
(Unramified) algebraic K-theory $\KK^{\Q}_*$ is a homotopy module, which arises from the algebraic K-theory spectrum $\KGL$. This is Theorem 3.13 of the seminal paper \cite{MR1813224} of Morel and Voevodsky.
\end{ex}
\begin{ex}
(Unramified) Hermitian K-theory $\KK_*^{\HH} = \GGW^*_*$ is a homotopy module, which arises from the Hermitian K-theory spectrum $\KQ$ constructed by Hornbostel in \cite{MR2122220}. A different model for this spectrum under the name $\BO$ was more recently constructed by Panin and Walter \cite{MR3882540}.
\end{ex}
By \cite{MR4282798} or by Chapter 2.3 of \cite{MR2934577}, homotopy modules come with residue and specialization maps, which also satisfy the properties of Proposition \ref{Properties of residue und specialization maps}. Furthermore, it is shown that the following two properties hold:
\begin{prop}\label{[t]x = 0}
For any homotopy module $M_*$ and any transcendental element $t$ over $k$, the map $M_*(k) \rightarrow M_{*+1}(k(t))$, $x \mapsto [t]x$ is injective with left-inverse $\partial^t_{\nu_t}$. In particular, if $[a]x = 0$ for all field extensions $k \subset F$ and all $a \in F$, then $x = 0$.
\end{prop}
\begin{prop}\label{Milnors exact sequence for homotopy modules}
For any homotopy module $M_*$ and any field extension $k \subset F$, there is a split short exact sequence
\begin{center}
\begin{tikzcd}
0 \arrow[r] & M_*(F) \arrow[r, "i_*"] & M_*(F(t))  \arrow[r, "\oplus_p \partial^p_{\nu_p}"] & \bigoplus\limits_{p \in \mathbb{A}^1} M_*(\kappa(v_p))  \arrow[r] & 0
\end{tikzcd}
\end{center}
of graded $M_*(F)$-modules, where $i_*$ is the map induced by the inclusion $F \subset F(t)$.
\end{prop}
These properties will be used to compute operations on generators of Milnor-Witt K-theory in Section 7.
\section{Warmup: The Additive and \texorpdfstring{$\mathbb{G}_m$-stable}{Gm-stable} Operations}\label{Section 5}
The results of this Section are essentially known and can be deduced from the results in \cite{MR2934577}, but their proofs are not recorded very well. Since these are not long, we decided to collect them here. The key ingredient is Theorem 3.37 of \cite{MR2934577}, which we will recall:
\begin{thm}[Morel]\label{Positive degree K_n^MW is the free strongly A^1-invariant abelian sheaf on the n-th smash power of G_m}
Let $M$ be a strictly $\mathbb{A}^1$-invariant abelian sheaf and let $n$ be a positive integer. The map $$u^* \colon \Hom_{\Ab_{\mathbb{A}^1} \! / k}(\KK^{\MW}_n,M) \rightarrow \Hom_{\Set_* \! / k}(\Gn,M)$$ induced by the universal symbol $u \colon \Gn \rightarrow \KK^{\MW}_n$, $(a_1,\dotsc,a_n) \mapsto [a_1,\dotsc,a_n]$ is a natural bijection in $M$. In other words, $\KK^{\MW}_n$ is the free strictly $\mathbb{A}^1$-invariant abelian sheaf $\tilde{\Z}_{\mathbb{A}^1}[\Gn]$ on the sheaf of pointed sets $\Gn$.
\end{thm}

\begin{cor}\label{Additive Operations on positive Milnor-Witt K-theory}
Let $n$ be a positive integer and let $M_*$ be a homotopy module. For all integers $m$, the abelian sheaf $\HHom_{\Ab_{\mathbb{A}^1} \! / k}(\KK^{\MW}_n,M_m)$ is isomorphic to $M_{m-n}$. In particular, we have $\HHom_{\Ab_{\mathbb{A}^1} \! / k}(\KK^{\MW}_n,\KK^{\MW}_m) \cong \KK^{\MW}_{m-n}$ for all integers $m$.
\end{cor}
\begin{proof}
Using Theorem \hyperref[Positive degree K_n^MW is the free strongly A^1-invariant abelian sheaf on the n-th smash power of G_m]{\ref{Positive degree K_n^MW is the free strongly A^1-invariant abelian sheaf on the n-th smash power of G_m}}, we have $\KK_n^{\MW} = \tilde{\Z}_{\mathbb{A}^1}[\Gn] \cong \tilde{\Z}_{\mathbb{A}^1}[\G]^{\otimes n}$. Therefore we get $\HHom_{\Ab_{\mathbb{A}^1} \! / k}(\KK^{\MW}_n,M_m) \cong \HHom_{\Ab_{\mathbb{A}^1} \! / k}(\Z,(M_m)_{-n}) \cong (M_m)_{-n}$ via the hom-tensor adjunction. Since $M_*$ is a homotopy module, this is just $M_{m-n}$.
\end{proof}
If one keeps track of the isomorphisms, it is not difficult to see that the isomorphism $K^{\MW}_{m-n}(k) \iso \Hom_{\Ab_{\mathbb{A}^1} \! / k}(\KK^{\MW}_n,\KK^{\MW}_m)$ maps an element $x$ to the multiplication with $x$. For general homotopy modules, it maps an element to the action of $\KK^{\MW}_n$ on said element.
\begin{cor}
Let $M_*$ be a homotopy module. For all integers $m$, the abelian sheaf $\HHom_{\Ab_{\mathbb{A}^1} \! / k}(\KK^{\MW}_0,M_m)$ is isomorphic to $M_m \oplus \prescript{}{h} M_{m-1}$. In particular, we have an isomorphism $\HHom_{\Ab_{\mathbb{A}^1} \! / k}(\KK^{\MW}_0,\KK^{\MW}_m) \cong \KK^{\MW}_m \oplus \,{\prescript{}{h} \KK^{\MW}_{m-1}}$ for all integers $m$.
\end{cor}
\begin{proof}
We have an isomorphism $\GGW \iso \Z \oplus \II$ by splitting of the rank, which we can translate to an isomorphism $\KK^{\MW}_0 \iso \Z \oplus \KK^{\MW}_1 \! \! / h$ on the level of Milnor-Witt K-theory. This gives 
$$\HHom_{\Ab_{\mathbb{A}^1} \! / k}(\KK^{\MW}_0,M_m) \cong \HHom_{\Ab_{\mathbb{A}^1} \! / k}(\Z,M_m) \oplus \HHom_{\Ab_{\mathbb{A}^1} \! / k}(\KK^{\MW}_1 \! \! / h,M_m)$$
with the first summand being $M_m$. The latter one is the kernel of 
$$h^* \colon \HHom_{\Ab_{\mathbb{A}^1} \! / k}(\KK^{\MW}_1,M_m) \rightarrow \HHom_{\Ab_{\mathbb{A}^1} \! / k}(\KK^{\MW}_1,M_m),$$
which under the isomorphism $\HHom_{\Ab_{\mathbb{A}^1} \! / k}(\KK^{\MW}_1,M_m) \cong M_{m-1}$ from the previous Corollary is ${\prescript{}{h} M_{m-1}}$ as claimed.
\end{proof}
It is also not difficult to keep track of the isomorphisms in this case:
$$\Hom_{\Ab_{\mathbb{A}^1} \! / k}(\KK^{\MW}_0,\KK^{\MW}_m) =  \rk \cdot M_m(k) \oplus (\langle - \rangle \mapsto [-]) \cdot {\prescript{}{h} M}_{m-1}(k)$$
Note that the latter map is not well-defined by itself. It really requires an element in the kernel of $h$. For $M_* = \KK^{\MW}_*$, the multiplication with an element $x \in K^{\MW}_m(k)$ is given by $\rk \cdot x+ (\langle - \rangle \mapsto [-])\eta x$, which allows us to write
$$\Hom_{\Ab_{\mathbb{A}^1} \! / k}(\KK^{\MW}_0,\KK^{\MW}_m) = \id \cdot K^{\MW}_m(k) \oplus (\langle - \rangle \mapsto [-]) \cdot {\prescript{}{h} \K^{\MW}_{m-1}(k)}.$$
Multiplication with a fixed element of suitable degree does of course also give us operations $K_{-n}^{MW} \rightarrow K_m^{\MW}$ and other operations on negative degree can be computed  via the previous Corollary. In light of Theorem 6.13 of \cite{MR2934577}, we see that the isomorphism $$\KK_{n-1}^{\MW} \cong \HHom_{\Ab_{\mathbb{A}^1} \! / k}(\KK^{\MW}_1,\KK_n^{\MW}) = \HHom_{\Ab_{\mathbb{A}^1} \! / k}(\tilde{\Z}_{\mathbb{A}^1}[\G],\KK_n^{\MW}) = (\KK_n^{\MW})_{-1}$$
given by multiplication is exactly the $\G$-suspension isomorphism. Therefore Corollary \ref{Additive Operations on positive Milnor-Witt K-theory} yields that a $\G$-stable operation of degree $m$ of Milnor-Witt K-theory needs to be a constant sequence of multiplications with a fixed element $x \in K^{\MW}_m(k)$, which certainly are well-defined operations. Let us record this observation:
\begin{cor}
The $\G$-stable operations of degree $m$ on Milnor-Witt K-theory are exactly the constant sequences $(x \cdot \id)_{n \in \mathbb{Z}}$ with $x \in K^{\MW}_m(k)$.
\end{cor}
\section{The Operations \texorpdfstring{$\lambda^n_l$}{λnl}}\label{Section 6}
We will now introduce the operations which ``essentially generate" all operations on Milnor-Witt K-theory. Here ``essentially" will mean that we have to allow certain infinite linear combinations, which we will explain later. For the entire section we let $n$ be a positive integer and all natural transformations/operations are considered to be between $\Set$-valued functors. If $k \subset F$ is a field extension and 
$$x = [a_{1,1},\dotsc,a_{1,n}] + \dotsc + [a_{r,1},\dotsc,a_{r,n}]\in K_n^{\MW}(F)$$
is a sum of pure symbols, we call  
$$\lambda^n_l(x) = \sum_{1 \leq i_1 < \dotsc < i_l \leq r} [a_{i_1,1},\dotsc,a_{i_1,n}] \cdot \dotsc \cdot [a_{i_l,1},\dotsc,a_{i_l,n}] \in K_{ln}^{\MW}(F)$$
the $l$-th divided power of $x$, where we allow $l$ to be any non-negative integer.
Since Milnor-Witt K-theory is not commutative, this does certainly not yield a well-defined map $K_n^{\MW}(F) \rightarrow K_{ln}^{\MW}(F)$ in general. However, if we let $\delta_n$ be $1$ if $n$ is odd and $0$ if $n$ is even, the following will turn out to give us divided power maps:
\begin{prop}\label{Definition of Lambda}
Let $k \subset F$ be a field extension and let $M_*$ be a homotopy module. Furthermore, let $S_n(F)$ be the set of symbols $\eta^d[a_1,\dotsc,a_{d+n}]$, where $d$ is a non-negative integer and $a_1,\dotsc,a_{d+n} \in F^\times$. For any $y \in \prescript{}{\delta_nh} M_*(k)$, the map
$$\Lambda^n \cdot y \colon \mathbb{Z}^{\oplus S_n(F)} \rightarrow M_*(F)[[t]]$$ 
given by mapping $\sum_{i = 1}^r m_i \eta^{d_i}[a_{i,1},\dotsc,a_{i,d_i+n}]$ to
$$\prod_{i = 1}^r \prod_{J \subset \lbrace 1, \dotsc, d_i + 1 \rbrace } \biggl( 1 + \Big[ \prod_{j \in J}a_{i,j},a_{i,d_i+2},\dotsc,a_{i,d_i+n} \Big]t \biggr)^{\hspace*{-4pt}(-1)^{e_{d_i,J}}\cdot m_i} \hspace*{-45pt} \cdot y,$$
where $e_{d_i,J} = d_i +1 - \lvert J \rvert$, is well-defined and factorizes through the quotient map $\mathbb{Z}^{\oplus S_n(F)} \twoheadrightarrow K_n^{\MW}(F)$.
\end{prop}
\begin{proof}
If $n$ is odd, we have $y = \langle 1 \rangle \cdot y = \epsilon \cdot y$ since $y \in \prescript{}{h} M_*(k)$. Therefore the products in $\Lambda^n \cdot y$ are independent of their order, which results in the well-definedness of this map. If $n$ is even, $\Lambda^n \cdot y$ is well-defined without restrictions on $y$ by the fact that $K^{\MW}_{2*}$ is commutative. Therefore we are either way in a commutative setting and will from now on freely change the order within the occuring products.

To show that $\Lambda^n \cdot y$ factorizes through the quotient map $\mathbb{Z}^{\oplus S_n(F)} \twoheadrightarrow K_n^{\MW}(F)$, we need to verify that all generators of the three relations from Lemma \ref{Standard presentation of positive degree Milnor-Witt K-theory} act trivially on $y$. Let $\eta^d[a_1,\dotsc,a_{d+n}] \in \mathbb{Z}^{\oplus S_n(F)}$ be a generator of the Steinberg relation, i.e. we have that $a_{i+1} = 1 - a_i$ for some $1 \leq i \leq d+n-1$. If $i \ leq d$, we can permute the $a_j$'s in the image and may thus assume that $i = 1$. Denoting the tuple $(a_{d+2},\dotsc,a_{d+n})$ by $a_{\underline{d}}$, the product
$$\Lambda^n \cdot y(\eta^d[a_1,\dotsc,a_{d+n}]) = \prod_{J \subset \lbrace 1,\dotsc, d+1 \rbrace} \biggl( 1 + \Big[ \prod_{j \in J}a_j,a_{\underline{d}} \Big]t \biggr)^{\hspace*{-4pt}(-1)^{e_{d,J}}}\hspace*{-29pt} \cdot y$$
can be rewritten as
$$\prod_{I \subset \lbrace 3,\dotsc, d+1 \rbrace} \frac{\left( 1 + \Big[ \prod_{i \in I}a_i,a_{\underline{d}} \Big]t \right)^{\hspace*{-2pt}(-1)^{e_{d,I}}}\left( 1 + \Big[ a_1a_2\prod_{i \in I}a_i,a_{\underline{d}} \Big]t \right)^{\hspace*{-2pt}(-1)^{e_{d,I}}}}{\left( 1 + \Big[ a_1\prod_{i \in I}a_i,a_{\underline{d}} \Big]t \right)^{\hspace*{-2pt}(-1)^{e_{d,I}}}\left( 1 + \Big[ a_2\prod_{i \in I}a_i,a_{\underline{d}} \Big]t \right)^{\hspace*{-2pt}(-1)^{e_{d,I}}}} \cdot y.$$
It therefore suffices to show that
$$(1+[b]t)(1+[a(1-a)b]t) = (1+[(1-a)b]t)(1-[ab]t)$$
for all $a,b \in F^\times$, where we are in a commutative setting. This amounts to showing the equality of the linear and quadratic coeffients of both sides. Using the Steinberg relation, the linear coefficient on the left hand side is
$$[b] + [a(1-a)b] = [b] + [1-a] + [ab] + \eta[1-a,ab] = [b] + [1-a] + [ab] + \eta[1-a,b],$$ 
which coincides with $[(1-a)b] + [ab]$, the one from the right hand side. The quadratic one on the left hand side is
$$[b,a(1-a)b] = [b,1-a] + [b,ab] + \eta[b,1-a,ab] = [b,1-a] + [b,ab] + \eta[b,1-a,b],$$
whereas the quadratic coefficient on the right hand side is
$$[(1-a)b,ab] = \eta[1-a,b,ab] + [1-a,ab] + [b,ab] = \eta[1-a,b,b] + [1-a,b] + [b,ab].$$
Since we are in a commutative setting, these two agree. If $i = d+1$, the same argument works, but one cannot ignore the contributions of $a_{\underline{d}}$. Finally, if $i \geq d+2$, the statement is clear. Therefore the map 
$\Lambda^n \cdot y$ factorizes through the quotient map $\mathbb{Z}^{\oplus S_n(F)} \twoheadrightarrow \mathbb{Z}^{\oplus S_n(F)}/ R_{\st}$, where $R_{\st}$ is the subgroup defined by the Steinberg relation. By abuse of notation we still denote the induced map on the quotient $\mathbb{Z}^{\oplus S_n(F)}/ R_{\st} \rightarrow M_*(F)[[t]]$ by $\Lambda^n \cdot y$. Let us now verify that the twisted tensor relation is respected. For this we consider a generator 
\begin{align*}
\eta^d[a_1,\dotsc,a_{i-1},bb',a_{i+1},\dotsc,a_{d+n}] &- \eta^d[a_1,\dotsc,a_{i-1},b,a_{i+1},\dotsc,a_{d+n}] \\ &- \eta^d[a_1,\dotsc,a_{i-1},b',a_{i+1},\dotsc,a_{d+n}] \\ &- \eta^{d+1}[a_1,\dotsc,a_{i-1},b,b',a_{i+1},\dotsc,a_{d+n}]
\end{align*}
of the twisted tensor relation in $\mathbb{Z}^{\oplus S(F)}/ R_{\st}$ and set $D_i = \lbrace 1,\dotsc, d+1 \rbrace \setminus \lbrace i \rbrace$. Using that we are in a commutative setting, we may once again assume that $i = 1$ since the case $i \geq d+2$ is trivial. Furthermore, we can ignore the contribution of the tuple $(a_{d+2},\dotsc,a_{d+n}) = a_{\underline{d}}$ as seen above. This reduces the task to showing that the product of
$$\prod_{J \subset D_i} \frac{\left( 1 + \Big[ \prod_{j \in J}a_j \Big]t \right)^{\hspace*{-2pt}(-1)^{e_{d,J}}}\hspace*{-6pt}\left( 1 + \Big[ b\prod_{j \in J}a_j \Big]t \right)^{\hspace*{-2pt}(-1)^{e_{d,J}}}\hspace*{-6pt}\left( 1 + \Big[ b'\prod_{j \in J}a_j \Big]t \right)^{\hspace*{-2pt}(-1)^{e_{d,J}}}}{\left( 1 + \Big[ bb'\prod_{j \in J}a_j \Big]t \right)^{\hspace*{-2pt}(-1)^{e_{d,J}}}\hspace*{-6pt}\left( 1 + \Big[ \prod_{j \in J}a_j \Big]t \right)^{\hspace*{-2pt}(-1)^{e_{d,J}}}\hspace*{-6pt}\left( 1 + \Big[ \prod_{j \in J}a_j \Big]t \right)^{\hspace*{-2pt}(-1)^{e_{d,J}}}}$$
and
$$\prod_{J \subset D_i} \frac{\left( 1 + \Big[ b\prod_{j \in J}a_j \Big]t \right)^{\hspace*{-2pt}(-1)^{e_{d+1,J}}}\left( 1 + \Big[ b'\prod_{j \in J}a_j \Big]t \right)^{\hspace*{-2pt}(-1)^{e_{d+1,J}}}}{\left( 1 + \Big[ \prod_{j \in J}a_j \Big]t \right)^{\hspace*{-2pt}(-1)^{e_{d+1,J}}}\left( 1 + \Big[ bb'\prod_{j \in J}a_j \Big]t \right)^{\hspace*{-2pt}(-1)^{e_{d+1,J}}}}$$
is $1$, which it clearly is. This gives us an induced map $\mathbb{Z}^{\oplus S_n(F)}/ R_{\st,\tt } \rightarrow M_*(F)[[t]],$ which we will still denote by $\Lambda^n \cdot y$. Here $R_{\st,\tt}$ is the subgroup defined by the generators of the Steinberg and twisted tensor relation. Finally, let us check the Witt relation. We pick a generator 
$$\eta^{d+2}[a_1,\dotsc,a_{i-1},-1,a_{i+1},\dotsc,a_{d+2+n}] +2\eta^{d+1}[a_1,\dotsc,a_{i-1},a_{i+1},\dotsc,a_{d+2+n}]$$ 
considered as an element of $\mathbb{Z}^{\oplus S_n(F)}/ R_{\st,\tt}$ and set $D_i = \lbrace 1,\dotsc, d+3 \rbrace \setminus \lbrace i \rbrace$ and $a_{\underline{d}} = (a_{d+4},\dotsc,a_{d+2+n})$. As before, we can reduce to the case that $i = 1$. After cancellation, this generator is now mapped to 
$$\prod_{J \subset D_i} \biggl( 1 + \Big[ -\prod_{j \in J}a_j,a_{\underline{d}} \Big]t \biggr)^{\hspace*{-4pt}(-1)^{e_{d+1,J}}}\hspace*{-5pt}\biggl( 1 + \Big[ \prod_{j \in J}a_j,a_{\underline{d}} \Big]t \biggr)^{\hspace*{-4pt}(-1)^{e_{d+1,J}}} \hspace*{-37pt} \cdot y,$$
which agrees with
$$\prod_{J \subset D_i} \biggl(1 + \biggl( \Big[ \prod_{j \in J}a_j,a_{\underline{d}} \Big] + \Big[ -\prod_{j \in J}a_j,a_{\underline{d}} \Big] \biggr) t \biggr)^{\hspace*{-4pt}(-1)^{e_{d+1,J}}} \hspace*{-37pt} \cdot y$$
by the fact that $[a,-a] = 0$ for all $a \in F^\times$. We also have 
$$[a] + [-a] = [a] + [-1] + [a] + \eta[a,-1] = [-1] + [a](2 + \eta[-1]) = [-1] + [a]h$$
for all $a \in F^\times$, which yields 
$$\Big[ \prod_{j \in J}a_j,a_{\underline{d}} \Big] + \Big[ -\prod_{j \in J}a_j,a_{\underline{d}} \Big] = [-1,a_{\underline{d}}]  + \Big[\prod_{j \in J}a_j,a_{\underline{d}} \Big]h = [-1,a_{\underline{d}}]  + h\sum_{j \in J}[a_j,a_{\underline{d}}]$$
for all $J \subset D_i$. The later summand does not contribute outside of degree $1$ since $[-1] \in \ker(h)$ and $[a,a] = [a,-1]$ for all $a \in F^\times$. Thus we are left with
$$\frac{\sum_{d-\lvert J \rvert \text{ even}}\sum_{j \in J}h[a_j,a_{\underline{d}}]t+\prod_{d-\lvert J \rvert \text{ even}} \left( 1 + [-1,a_{\underline{d}}]t \right)}{\sum_{d-\lvert J \rvert \text{ odd}}\sum_{j \in J}h[a_j,a_{\underline{d}}]t+\prod_{d-\lvert J \rvert \text{ odd}}\left( 1 + [-1,a_{\underline{d}}]t \right)} \cdot y,$$
which by a simple counting argument is $1 \cdot y = y$. This finishes the proof.
\end{proof}
From the definition of $\Lambda^n \cdot y$ it is clear that this map is functorial in $F$ since $y$ is defined over $k$. Not clear may be how $\Lambda^n \cdot y$ relates to the divided powers as introduced before. 
\begin{de}\label{Definition of lamda}
Let $l$ be a non-negative integer and let $y \in \prescript{}{\delta_nh} M_*(k)$ for some homotopy module $M_*$. The $l$-th divided power operation on $K_n^{\MW}$ associated to $y$ is the operation $K_n^{\MW} \rightarrow M_*$ given by taking the coefficient of $\Lambda^n \cdot y(x)$ of degree $l$ for all elements $x \in K_n^{\MW}(F)$ and all field extensions $k \subset F$. 
\end{de}
We denote the $l$-th divided power operation on $K_n^{\MW}$ associated to $y$ by $\lambda^n_l \cdot y$, which as before is not only a notation, but allows us to work with the non-defined operation $\lambda^n_l$ as long as we act on $y \in \prescript{}{\delta_nh} M_*(k)$ in the end. Of course, $\lambda^n_0 = 1$ and $\lambda^n_1 = \id$ are defined for all $y$. Furthermore we will just refer to an $l$-th divided power on $K_n^{\MW}$ when speaking about $\lambda^n_l \cdot y$ for some homotopy module $M_*$ and $y \in \prescript{}{\delta_nh} M_*(k)$. By its definition, we get:
\begin{prop}\label{lambda of sum}
Let $k \subset F$ be a field extension. We have 
$$\lambda^n_l \cdot y(x+x') = \sum_{i = 0}^l \lambda^n_i(x)\lambda^n_{l-i}(x') \cdot y$$
for all elements $x,x' \in K_n^{\MW}(F)$.
\end{prop}
\begin{cor}\label{lambda is divided power}
Let $k \subset F$ be a field extension. If 
$$x = [a_{1,1},\dotsc,a_{1,n}] + \dotsc + [a_{r,1},\dotsc,a_{r,n}]\in K_n^{\MW}(F)$$ 
is a sum of pure symbols, then 
$$\lambda^n_l \cdot y(x) = \sum_{1 \leq i_1 < \dotsc < i_l \leq r} [a_{i_1,1},\dotsc,a_{i_1,n}] \cdot \dotsc \cdot [a_{i_l,1},\dotsc,a_{i_l,n}] \cdot y.$$
\end{cor}
This justifies the name and also explains why $\Lambda^n \cdot y$ is defined the way it is. An arbitrary element of $K^{\MW}_n(F)$ is first rewritten in terms of pure symbols and then one extends the desired formula from the previous corollary via Proposition \hyperref[lambda of sum]{\ref{lambda of sum}} to negative signs. The element $y$ is still needed for it to map to $M_*$ and to be well-defined in the case of odd $n$, of course.
\section{Shifts for Operations on Milnor-Witt K-theory}\label{Section 7}
In this section we will state the basic tools needed for our computations later. For these we will consider certain operations on Milnor K-theory. Let $M_*$ be a homotopy module and let $y \in M_*(k)$. For a positive integer $n$ and a non-empty ordered subset $\lbrace i_1, \dotsc,i_l \rbrace \subset \lbrace 1,\dotsc,n \rbrace$, we define the operation 
$$[-_{i_1},\dotsc,-_{i_l}] \cdot y \colon (K^{\MM}_1)^n \rightarrow M_*$$
by mapping tuples $(a_1,\dotsc,a_n) \in \Gm_m^n(F) \cong (K^{\MM}_1(F))^n$ to $[a_{i_1},\dotsc,a_{i_l}] \cdot y$ for every field extension $k \subset F$. Furthermore we set this operation to be the constant operation with value $1$ in the case that $l = 0$. These operations clearly commute with specialization maps and turn out to generate all such operations $(K^{\MM}_1)^n \rightarrow M_*$, essentially by Theorem 3.18 of Vial \cite{MR2497581} with minor adaptations to generalize to homotopy modules:
\begin{thm}[Vial]\label{Operations on (K_1^M)^r}
Let $M_*$ be a homotopy algebra and let $n$ be a positive integer. The $M_*(k)$-module $\Op_{\sp}((K^{\MM}_1)^n, M_*)$ of operations $(K^{\MM}_1)^n \rightarrow M_*$ commuting with specialization maps is given by the free $M_*(k)$-module $$\bigoplus_{l=0}^n \bigoplus_{1 \leq i_1 < \dotsc < i_l \leq n} [-_{i_1},\dotsc,-_{i_l}] \cdot M_*(k).$$
\end{thm}
\begin{proof}
This can indeed be proven like Theorem 3.18 and the third step of Theorem 3.4 from \cite{MR2497581}. Since the given proof does not contain all details, we will give the full argument for the convenience of the reader and begin with the case $n = 1$.

\medskip

Step 1: An operation $\varphi \colon K^{\MM}_1 \rightarrow M_*$ commuting with specialization maps is determined by its image at a single transcendental element over $k$.

Let $\varphi \in \Op_{\sp}(K^{\MM}_1, M_*)$, let $F$ and $F'$ be a field extensions of $k$ and let $t \in F$ be a transcendental element over $k$. Furthermore let $a \in F'$. If $a$ is transcendental over $k$, we have a canonical isomorphism $f \colon k(t) \rightarrow k(a)$ given by substituting $a$ for $t$, so that $\varphi(\lbrace a \rbrace) \in M_*(k(a))$ is determined by $\varphi(\lbrace t \rbrace)  \in M_*(k(t))$ via the induced isomorphism $f_* \colon M_*(k(t)) \rightarrow M_*(k(a))$. If $a$ is algebraic over $k$, we pick a transcendental element $u$ over $F'$ and consider the square
\begin{center}
\begin{tikzcd}
 K^{\MM}_1(F'(u)) \arrow[r, "\varphi"] \arrow[d, "s_{\nu_u}^u"] & M_*(F'(u)) \arrow[d, "s_{\nu_u}^u"]\\
 K^{\MM}_1(F') \arrow[r, "\varphi"] & M_*(F')
\end{tikzcd}
\end{center}
which by our assumptions on $\varphi$ is commutative. We get
$$\varphi(\lbrace a \rbrace) = \varphi(s_{\nu_u}^u(\lbrace au \rbrace)) = s_{\nu_u}^u(\varphi(\lbrace au \rbrace))$$
and $au \in F'(u)$ is transcendental over $k$, so that $\varphi(\lbrace a \rbrace)$ is once again determined by $\varphi(\lbrace t \rbrace)$. 

\medskip

Step 2: The image of $\varphi$ at our fixed transcendental element $t$ is unramified at all closed points of $\mathbb{P}^1_{k(t)} \! \setminus \lbrace 0, \infty \rbrace$. 

Let $u$ be another transcendental element over $k$ and consider the field extension $k(u)$ of $k$. We now pick a monic irreducible polynomial $p \in k[t]$ with $p \neq t$ and view it as a polynomial with coefficients in $k(u)$ via the  inclusion $k \hookrightarrow k(u)$. The element $-up \in k(u,t)$ is then a uniformizer for the corresponding discrete valuation $\nu_p$ on $k(u,t)$ and we consider the commutative square
\begin{center}
\begin{tikzcd}
 K^{\MM}_1(k(u,t)) \arrow[r, "\varphi"] \arrow[d, "s_{\nu_p}^{-up}"] & M_*(k(u,t)) \arrow[d, "s_{\nu_p}^{-up}"]\\
 K^{\MM}_1(\kappa_p(u)) \arrow[r, "\varphi"] & M_*(\kappa_p(u)).
\end{tikzcd}
\end{center}
We denote by $i$ the inclusion $k(t) \hookrightarrow k(u,t)$. Since $\varphi$ and, by Lemma 3.16 of \cite{MR2497581} also $i_*$, commute with specialization maps, we now have
$$s_{\nu_p}^{-up}(\varphi(i_*(\lbrace t \rbrace))) = s_{\nu_p}^{-up}(i_*(\varphi(\lbrace t \rbrace))) = i_*(\varphi(s_{\nu_p}^{-up}(\lbrace t \rbrace))) = i_*(\varphi(\lbrace \overline{t} \rbrace )),$$
which is also given by
$$s^p_{\nu_p}(i_*(\varphi(\lbrace t \rbrace))) + \epsilon [ \overline{-u}] \partial^{p}_{\nu_p}(i_*(\varphi(\lbrace t \rbrace))) = i_*(\varphi(\lbrace \overline{t} \rbrace)) + \epsilon [ \overline{-u}] \partial^{p}_{\nu_p}(i_*(\varphi(\lbrace t \rbrace))).$$
according to Lemma \ref{Properties of residue und specialization maps}. Applying $\partial_{\nu_{-u}}^{-u} \colon M_*(\kappa_p(u)) \rightarrow M_{*-1}(\kappa_p)$ thus yields
$$\epsilon i_*(\partial_{\nu_p}^p(\varphi(\lbrace t \rbrace))) = \epsilon \partial_{\nu_p}^p(i_*(\varphi(\lbrace t \rbrace))) = 0$$ and therefore $i_*(\partial_{\nu_p}^p(\varphi(\lbrace t \rbrace))) = 0$. Hence $\partial_{\nu_p}^p(\varphi(\lbrace t \rbrace)) = 0$ by the injectivity of $i_*$. In other words, $\varphi(\lbrace t \rbrace)$ is unramified at all closed points of $\mathbb{P}^1_{k(t)} \! \setminus \lbrace 0, \infty \rbrace$ as claimed.

Now we put everything together. We set $a = \partial^{\nu_t}_t(\varphi(\lbrace t \rbrace)) \in M_*(k)$ and consider the difference $\varphi(\lbrace t \rbrace) - [t]a \in M_*(k(t))$. This element is by construction unramified on $\mathbb{P}^1_{k(t)} \! \setminus \lbrace \infty \rbrace$, which by Milnor's exact sequence means that it is given by an element $b \in M_*(k)$ considered as an element in $M_*(k(t))$. Hence we have $\varphi(\lbrace t \rbrace) = [t]a + b$ for some suitable elements $b \in M_*(k)$ and $a \in M_*(k)$. Combining this with our first step, we see that there exist elements $a \in M_*(k)$ and $b \in M_*(k)$, so that $\varphi = [-]a + b$. Since assignments of this form certainly define operations, this finishes the proof for $n = 1$.

\medskip

Step 3: The general case.

We now conclude by induction on $n \geq 1$. The case $n = 1$ has already been treated. Let us therefore assume that the claim is true for all positive integers $l \leq n$ for some $n \geq 1$. Let $\varphi \in \Op_{\sp}((K^{\MM}_1)^{n+1}, M_*)$, let $k \subset F$ be a field extension and let $x \in (K_1^{\MM}(F))^n$ be a fixed element. Then $\varphi((x,-)) \colon K^{\MM}_1 \rightarrow M_*$ defines an operation over the field $F$, which we will denote by $\varphi_x$. By the previous step, there exist elements $a_x, b_x \in M_*(F)$ such that $\varphi_x = [-]a_x + b_x$. The assignments $x \mapsto a_x$ and $x \mapsto b_x$ define operations in $\Op_{\sp}((K^{\MM}_1)^n, M_*)$ and are by induction hypothesis hence given by $M_*(k)$-linear combinations of $[-_{i_1},\dotsc,-_{i_l}] \cdot 1$ for $0 \leq l \leq n$ and $1 \leq i_1 < \dotsc < i_l \leq n$. Together with  $\varphi_x = [-]a_x + b_x$ this implies that $\varphi$ is of the claimed form. It remains to show that the operations of the form $[-_{i_1},\dotsc,-_{i_l}] \cdot 1$ are linearly independent. Suppose 
$$\varphi = \sum_{l = 0}^n \sum_{1 \leq i_1 < \dotsc < i_l \leq n} [-_{i_1},\dotsc,-_{i_l}] \cdot \lambda_{i_1,\dotsc,i_l} = 0$$
where $\lambda_{i_1,\dotsc,i_l} \in M_*(k)$ for $0 \leq l \leq n$ and $1 \leq i_1 < \dotsc < i_l \leq n$. We fix one ordered subset $\lbrace j_1, \dotsc,j_s \rbrace \subset \lbrace 1, \dotsc, n \rbrace$ and consider the finitely generated field extension $k \subset k(t_{j_1},\dotsc,t_{j_s}) = F$. We set $t_j = 0$ for all $j \in \lbrace 1,\dotsc,n\rbrace \setminus \lbrace j_1,\dotsc,j_s \rbrace$ and let $\underline{t} = (t_1,\dotsc,t_n) \in (K^{\MM}_1(F))^n$. Then we have
$$\lambda_{j_1,\dotsc,j_s} =  \partial^{t_{j_s}}_{\nu_{t_{j_s}}} \circ \dotsc \circ \partial^{t_{j_1}}_{\nu_{t_{j_1}}}(\varphi(\underline{t})) = \partial^{t_{j_s}}_{\nu_{t_{j_s}}} \circ \dotsc \circ \partial^{t_{j_1}}_{\nu_{t_{j_1}}}(0) = 0,$$
which we had to show.
\end{proof}
\begin{rem}
This proof in particular shows that we do not need to distinguish between operations defined on all field extensions or only on finitely generated ones. Since Theorem \ref{Operations on (K_1^M)^r} is the very first ingredient of our computation, we will thus from now just speak about operations without specifying the underlying category of field extensions.
\end{rem}
We now consider the subfunctor $[-_1,\dotsc,-_n] \subset K^{\MW}_n$ which for every field extension $k \subset F$ is given by $[F^\times,\dotsc,F^\times]$, the pure symbols with entries from $F^\times$. Note that we use this notation both for this subfunctor and for the operations arising in the previous Theorem. According to Lemma \ref{Generators of positive Milnor-Witt K-theory groups}, this subfunctor encodes exactly the canonical generators of $K^{\MW}_n$ and our goal is to understand the operations on these generators. To express what it means for such operations to commute with specialization maps (which do not restrict to $[-_1,\dotsc,-_n]$), we do the following. Since $[-_1,\dotsc,-_n]$ is the image of the universal symbol  $u \colon \Gn \rightarrow K_n^{\MW}$, the operations $[-_1,\dotsc,-_n] \rightarrow M_*$ correspond to operations $\mathbb{G}_m^n \rightarrow M_*$ which factorize through $(K^{\MM}_1)^n = \mathbb{G}_m^n \twoheadrightarrow \Gn \rightarrow [-_1,\dotsc,-_n]$. The latter map will also be called universal symbol and denoted by $u$. We know what it means for the latter operations to commute with specialization maps and can restrict to those. This gives us a definition of $\Op_{\sp}([-_1,\dotsc,-_n],M_*)$, which we will now determine in light of the previous theorem.
\begin{thm}\label{Operations on (K_1^M)^r that factorize through the multiplication map}
For any homotopy algebra $M_*$ and any positive integer $n$, the $M_*(k)$-module $\Op_{\sp}([-_1,\dotsc,-_n],M_*)$ is free of rank $2$ generated by the constant operation $1$ and $[-_1,\dotsc,-_n] \cdot 1$.
\end{thm}
\begin{proof}
By definition, the $M_*(k)$-module $\Op_{\sp}([-_1,\dotsc,-_n],M_*)$ is the submodule of $\Op_{\sp}((K^{\MM}_1)^n  ,M_*)$ given by those operations which factorize through the universal symbol $u \colon (K^{\MM}_1)^n = \mathbb{G}_m^n \twoheadrightarrow \Gn \rightarrow [-_1,\dotsc,-_n]$. According to Theorem \ref{Operations on (K_1^M)^r} the $M_*(k)$-module $\Op_{\sp}((K^{\MM}_1)^n, M_*)$ is 
$$\bigoplus_{l=0}^n \bigoplus_{1 \leq i_1 < \dotsc < i_l \leq n} [-_{i_1},\dotsc,-_{i_l}]M_*(k)$$ and it is clear that its submodule $M_*(k) \oplus [-_1,\dotsc,-_n]M_*(k)$ consists of operations that factorize through $u$. It remains to show that these are the only ones, which we will do by induction on $n$. 

For $n = 1$ the statement coincides with the $n = 1$ case of Theorem \ref{Operations on (K_1^M)^r} and is thus already shown. We now assume that the statement is true up to some positive integer $n$. Let $\varphi \in \Op_{\sp}([-_1,\dotsc,-_{n+1}],M_*)$, let $k \subset F$ be a field extension and let $x \in (K^{\MM}_1(F))^n$. Then $\varphi_x = \varphi(x,-_{n+1})$ defines an operation $K^{\MM}_1 \rightarrow M_*$ defined over the field $F$, and is by the previous theorem hence given by $b_x + [-_{n+1}]b_x'$ for some elements $b_x, b_x' \in M_*(F)$. We now let $\psi$ and $\psi'$ denote the operations $(K^{\MM}_1)^n \rightarrow M_*$ over $k$, given by mapping $x \in (K^{\MM}_1(F))^n$ to $b_x$ and $b_x'$ respectively.

\medskip

Step 1: The operation $\psi$ is constant. In particular, the operation 
$[-_{n+1}]\psi'$ factorizes through $u$.

By the fact that $[1] = 0$ in Milnor-Witt K-theory, we have $[x,1] = [x',1]$ for all field extensions $k \subset F$ and all elements $x,x' \in (F^\times)^n = (K^{\MM}_1(F))^n$. Since the operation $\varphi = \psi + [-_{n+1}]\psi'$ factorizes through $u$, this gives us
$$\psi(x) = \psi (x) + [1]\psi'(x) = \varphi(x,1) = \varphi(x',1) = \psi (x') + [1]\psi'(x') = \psi(x').$$
In other words, the operation $\psi$ is constant. Therefore, if we consider $\psi$ as an operation $(K^{\MM}_1)^{n+1} \rightarrow M_*$, it factorizes through $u \colon (K^{\MM}_1)^{n+1} \rightarrow [-_1,\dotsc,-_{n+1}]$. Since the operations which factorize through $u$ are a $M_*(k)$-module and in particular a group, also $[-_{n+1}]\psi' = \varphi - \psi$ factorizes through $u$.

\medskip

Step 2: The operation $\psi'$ factorizes through $u \colon (K^{\MM}_1)^n \rightarrow [-_1,\dotsc,-_n]$.

Let $k \subset F$ be a field extension and let $a_1,\dotsc,a_n,a_1',\dotsc,a_n' \in (K^{\MM}_1(F))^n$ with $[a_1,\dotsc,a_n] = [a_1',\dotsc,a_n']$. Thus, if $t$ is a transcendental element over $F$, we have $[a_1,\dotsc,a_n,t] = [a_1',\dotsc,a_n',t] \in K_{n+1}^{\MW}(F(t))$. Since the operation $[-_{n+1}]\psi'$ facorizes through $u \colon (K^{\MM}_1)^n \rightarrow [-_1,\dotsc,-_n]$, we get $[t]\psi'(a_1,\dotsc,a_n) = [t]\psi'(a_1',\dotsc,a_n')$ and hence $\psi'(a_1,\dotsc,a_n) = \psi'(a_1',\dotsc,a_n')$ by Proposition \ref{[t]x = 0}.

\medskip

Step 3: The operation $\varphi$ is of the wanted form.

Using step 1 and 2 and the induction hypothesis, we know that there exist elements $x,y,z \in M_*(k)$ with
$$\varphi = x + [-_{n+1}](y + [-_1,\dotsc,-_n]z) = x + [-_{n+1}]y + [-_1,\dotsc,-_{n+1}] \epsilon^n z.$$
Renaming $\epsilon^n z = z' \in M_*(k)$, it remains to show that $y = 0$. Since both $\varphi$ and $\varphi - [-_{n+1}]y = x + [-_1,\dotsc,-_{n+1}] z'$ factorize through $u$, so does the operation $[-_{n+1}]y$. Let $k \subset F$ be a field extension and let $a_1,\dotsc,a_{n-1} \in F^\times$. Furthermore let $t$ be transcendental over $F$. Then we have 
$$[a_1,\dotsc,a_{n-1},t,t] = [a_1,\dotsc,a_{n-1},t,-1] \in K_{n+1}^{\MW}(F(t))$$ and hence $[t]y = [-1]y$. Applying the residue map $\partial_{\nu_t}^t$ now yields $y = 0$ since $[-1]y$ is defined over $F$.
\end{proof}
Now that we computed the operations on the generators of Milnor-Witt K-theory, we follow Garrel's strategy from \cite{MR4113769} and measure how an operation changes by adding or subtracting generators in order to understand all operations. For this we will now restrict to $\mathbb{N}$-graded homotopy algebras $M_*$. For these we consider the separated filtration given by $F_dM_* = M_{\geq d}$, where $d \geq 0$. Note that being $\mathbb{N}$-graded in particular gives us that the filtration pieces $F_dM_*$ define ideals in $M_*$. Furthermore the filtration endows the $M_*(k)$-module $\Op_{\sp}(K^{\MW}_n, M_*)$ with a separated filtration given by $$F_d\Op_{\sp}(K^{\MW}_n, M_*) = \Op_{\sp}(K^{\MW}_n, F_dM_*)$$ for all non-negative integers $d$.
\begin{rem}
Even though it is per se not an example of an $\mathbb{N}$-graded homotopy algebra, all of the following arguments will also work for the Witt ring $W$ together with the separated filtration given by powers of the fundamental ideal $I$ and the usual residue and specialization morphisms, as for example found in \cite{MR2427530}. Here the $K_*^{\MW}$-action is the multiplication action after passing to the quotient $K_*^{\W} \cong I^*$.
\end{rem}
Recall that $n$ is a positive integer. In the following proposition we will make use of the $n$-th negative shift of a filtration, which is commonly denoted by $[-n]$. We stress this to ensure that the reader does not confuse this shift with a symbol of Milnor-Witt K-theory. 
\begin{prop}\label{Shifts}
For all $\mathbb{N}$-graded homotopy algebras $M_*$, there exist unique morphisms 
$$\partial^{\pm} \colon \Op_{\sp}(K^{\MW}_n, M_*) \rightarrow \Op_{\sp}(K^{\MW}_n, M_*)[-n]$$
of filtered $M_*(k)$-modules, such that
$$\varphi(x \pm [\underline{a}]) = \varphi(x) \pm [\underline{a}]\partial^{\pm}(\varphi)(x)$$
for all $\varphi \in \Op_{\sp}(K^{\MW}_n, M_*)$, $x \in K^{\MW}_n(F)$, $\underline{a} \in (F^\times)^n$ and all field extensions $k \subset F$.
\end{prop}
\begin{proof}
Let $\varphi \in \Op_{\sp}(K^{\MW}_n, M_*)$. Furthermore let $k \subset F$ be a field extension, let $x \in K^{\MW}_n(F)$ and let $\underline{a} \in (L^\times)^n$ for some field extension $F \subset L$. 

We set $\psi(\varphi)_x([\underline{a}]) =  \varphi(x \pm [\underline{a}])$, which yields a operations 
$$\psi(\varphi)_x \in \Op_{\sp}([-_1,\dotsc,-_n], M_*)$$
defined over $F$. Theorem \ref{Operations on (K_1^M)^r that factorize through the multiplication map} now gives $\psi(\varphi)_x = [-_1,\dotsc,-_n]a_x + b_x$ for some $a_x,b_x \in M_*(F)$. Since $0 \in [F^\times,\dotsc,F^\times]$ we have
$$\varphi(x) = \varphi(x \pm 0) = \psi(\varphi)_x(0) = 0 \cdot a_x + b_x = b_x.$$
Setting $\partial^{\pm}(\varphi)(x) = a_x$ therefore does the job and also clarifies that $\partial^{\pm}$ is unique with the wanted property. Furthermore, $\partial^{\pm}$ is by definition clearly a morphism of $M_*(k)$-modules. It remains to verify that $\partial^{\pm}$ respects the respective filtrations. If $\varphi \in \Op_{\sp}(K^{\MW}_n, F_dM_*)$ for some integer $d$, then for any element $x \in K^{\MW}_n(F)$ we have $\varphi(x \pm [\underline{a}]) = [\underline{a}]\partial^{\pm}(\varphi)(x) \in F_dM_*(L)$ for all $\underline{a} \in (L^\times)^n$. Hence $\partial^{\pm}(\varphi)(x)$ lives in $F_{d-n}M_*(F)$, which finishes the proof.
\end{proof}
We will usually denote $\partial^\pm(\varphi)$ by $\varphi^{(\pm)}$ and refer to them as positive and negative shifts of $\varphi$. Let us record some useful examples:
\begin{prop}\label{Shift is derivative}
For all $\mathbb{N}$-graded homotopy algebras $M_*$ and all $y \in \prescript{}{\delta_nh} M_*(k)$ we have
\begin{enumerate}
\item[(i)] ${\lambda_0^n \cdot 1}^{(\pm)} = 0$ and ${\lambda_1^n \cdot 1}^{(\pm)} = \lambda_0^n \cdot 1$;
\item[(ii)] ${\lambda_l^n \cdot y}^{(+)} = \lambda_{l-1}^n \cdot y$ for all integers $l \geq 2$;
\item[(iii)] ${\lambda_l^n \cdot y}^{(-)} = \sum_{i = 0}^{l-1}(-1)^{l-(i+1)}[-1]^{n(l-(i+1))}\lambda_i^n \cdot y$ for all integers $l \geq 2$;
\end{enumerate}
\end{prop}
\begin{proof}
This is a direct consequence of Proposition \ref{lambda of sum}. For the convenience of the reader we will nevertheless quickly do the computations. Let $k \subset F$ be a field extension, let $x \in K^{\MW}_n(F)$ and let $\underline{a} \in (F^\times)^n$. We have
$$\lambda_0^n \cdot 1(x \pm [\underline{a}]) = 1 = \lambda_0^n \cdot 1(x) \text{ and }\lambda_1^n \cdot 1(x \pm [\underline{a}]) = (x \pm [\underline{a}]) \cdot 1 = \lambda_1^n \cdot 1(x) \pm [\underline{a}] \cdot 1,$$
which shows (i). Now let $l \geq 2$. The computation
$$\lambda_l^n \cdot y(x + [\underline{a}]) = \sum_{i + j = l}\lambda_i^n(x) \lambda_j^n([\underline{a}]) \cdot y = (\lambda_l^n(x) + \lambda^n_{l-1}([\underline{a}])) \cdot y = \lambda^n_l \cdot y(x) + [\underline{a}]\lambda^n_{l-1} \cdot y(x)$$
shows claim (ii) and the computation
$$\lambda_l^n \cdot y(x - [\underline{a}]) = \sum_{i + j = l}\lambda_i^n(x) \lambda_j^n(-[\underline{a}]) \cdot y =   \lambda^n_l \cdot y(x) + [\underline{a}]\sum_{i = 0}^{l-1}(-1)^{l-i}[-1]^{n(l-i-1)}\lambda_i^n \cdot y(x)$$ shows (iii) by rewriting the second summand as
$$[\underline{a}]\sum_{i = 0}^{l-1}(-1)^{l-i}[-1]^{n(l-i-1)}\lambda_i^n \cdot y(x) = -[a]\sum_{i = 0}^{l-1}(-1)^{l-(i+1)}[-1]^{n(l-(i+1))}\lambda_i^n \cdot y(x).$$
\end{proof}
It is natural to ask in which relation the two shifts are, especially when applying both of them on the same operation. For now we get the following, which we will improve later on.
\begin{lem}\label{Properties of derivatives}
Let $\varphi \in \Op_{\sp}(K^{\MW}_n, M_*)$ for some $\mathbb{N}$-graded homotopy algebra $M_*$. Then we have
\begin{enumerate}
\item[(i)] $(\varphi^{(+)})^{(-)} = \epsilon^n(\varphi^{(-)})^{(+)}$.
\item[(ii)] $(\varphi^{(+)})^{(+)} \in \Op_{\sp}(K^{\MW}_n, \prescript{}{\delta_nh} M_*)$;
\item[(iii)] $\varphi^{(+)} - \varphi^{(-)}= [-1]^n(\varphi^{(+)})^{(-)}$;
\end{enumerate}
\end{lem}
\begin{proof}
Let $\varphi \in \Op_{\sp}(K^{\MW}_n, M_*)$, let $k \subset F$ be a field extension, $x \in K_n^{\MW}(F)$ and $\underline{a},\underline{b} \in (F^\times)^n$. We can compute $\varphi(x + [\underline{a}] - [\underline{b}])$ in two ways. We get
$$\varphi(x + [\underline{a}] - [\underline{b}]) = \varphi(x) + [\underline{a}]\varphi^{(+)}(x) - [\underline{b}](\varphi^{(-)}(x) + [\underline{a}](\varphi^{(-)})^{(+)}(x))$$
by first applying the defining formula of $\partial^{-1}$ and then the one of $\partial^{+1}$ and
$$\varphi(x + [\underline{a}] - [\underline{b}]) = \varphi(x) - [\underline{b}]\varphi^{(-)}(x) + [\underline{a}](\varphi^{(+)}(x) - [\underline{b}](\varphi^{(+)})^{(-)}(x))$$
if we use the other order. Hence we have
$$[\underline{a},\underline{b}](\varphi^{(+)})^{(-)}(x) - [\underline{b},\underline{a}]\varphi^{(-)})^{(+)}(x) = 0.$$ Using $[\underline{b},\underline{a}] = \epsilon^{n^2}[\underline{a},\underline{b}] = \epsilon^n[\underline{a},\underline{b}]$ and choosing $F = k(\underline{t},\underline{s})$ and $\underline{a} = \underline{t}$ and $\underline{b} = \underline{s}$ for some transcendental elements $\underline{t}$ and $\underline{s}$ over $k$, we thus have $(\varphi^{(+)})^{(-)} = \epsilon^n(\varphi^{(-)})^{(+)}$ by Proposition \ref{[t]x = 0}, which shows (i). Similarly one gets
$$[\underline{a},\underline{b}]\delta_n h(\varphi^{(+)})^{(+)}(x) = 0,$$
which then by Proposition \ref{[t]x = 0} yields that $(\varphi^{(+)})^{(+)} \in \Op_{\sp}(K^{\MW}_n,\prescript{}{\delta_nh} M_* )$ as claimed in (ii). Finally, let us show (iii). Setting $\underline{a} = \underline{b}$, we get
$$\varphi(x) = \varphi(x) - [\underline{a}]\varphi^{(-)}(x) + [\underline{a}](\varphi^{(+)}(x) - [\underline{a}](\varphi^{(+)})^{(-)}(x))$$
from the second formula above and therefore 
$$[\underline{a}]([-1]^n(\varphi^{(+)})^{(-)}(x) - (\varphi^{(+)} - \varphi^{(-)})) = 0$$
by using that $[\underline{a},\underline{a}] = [\underline{a}][-1]^n$. Choosing $F = k(t_1,\dotsc,t_n)$ and $a_i = t_i$ for some transcendental elements $t_i$ over $k$ therefore once again completes the argument by Proposition \ref{[t]x = 0}.
\end{proof}
We will also consider quotients of $\Op_{\sp}(K^{\MW}_n, M_*)$. For typographical reasons, we will occasionally also denote operations from $K^{\MW}_n$ to some $\mathbb{N}$-graded homotopy algebra $M_*$ by $\Op^n(M_*)$ and the ones respecting specialization maps by $\Op^n_{\sp}(M_*)$.
\begin{prop}\label{ker of shifts}
For all $\mathbb{N}$-graded homotopy algebras $M_*$ and for all non-negative integers $d$, the morphisms $\partial^{\pm}$ induce morphisms
$$\Op^n_{\sp}(M_*)/\!\Op^n_{\sp}(F_{d+n}M_*) \xrightarrow{\overline{\partial^{\pm}}} \Op^n_{\sp}(M_*)/\!\Op^n_{\sp}(F_dM_*)$$
of $M_*(k)/F_dM_*(k)$-modules whose kernels are $M_*(k)/F_{d+n}M_*(k)$. In particular, the kernels of $\partial^{\pm}$ are the submodule $M_*(k)$ of constant operations.
\end{prop}
\begin{proof}
If $\varphi, \psi \in \Op^n_{\sp}(M_*)$ are two operations whose images $\overline{\varphi}$ and $\overline{\psi}$ in the quotient $\Op^n_{\sp}(M_*)/\!\Op^n_{\sp}(F_{d+n}M_*)$ coincide, then we also have $\overline{\varphi^{(\pm)}} = \overline{\psi^{(\pm)}}$ in $\Op^n_{\sp}(M_*)/\!\Op^n_{\sp}(F_{d}M_*)$ since $\partial^{\pm}$ maps $\Op^n_{\sp}(F_{d+n}M_*)$ to $\Op^n_{\sp}(F_dM_*)$ by Proposition \ref{Shifts}. Therefore $\overline{\partial^{\pm}}$ is a well-defined.
\\Let $\varphi \in \Op^n_{\sp}(M_*)$ with $\varphi^{(\pm)} \in \Op^n_{\sp}(F_dM_*)$. Furthermore let $k \subset F$ be a field extension, let $\underline{a} \in (F^\times)^n$ and let $x \in K^{\MW}_n(F)$. We have
$$\varphi(x \pm [\underline{a}]) = \varphi(x) \pm [\underline{a}]\varphi^{(\pm)}(x) = \varphi(x) \text{ mod } F_{d+n}M_*(F)$$
and know that every element $x \in K^{\MW}_n(F)$ can be written as a sum and difference of elements in $[F^\times,\dotsc, F^\times]$. Therefore we get
$$\varphi(x) = \varphi(0) \text{ mod } F_{d+n}M_*(F)$$ 
by repeating the previous computation, so that $\varphi \in M_*(k)/F_{d+n}M_*(k)$. Since such elements certainly are in the kernel of $\overline{\partial^\pm}$, this shows the first claim.
\\Now let $\varphi$ be in the kernel of $\partial^{\pm}$. Furthermore let $k \subset F$ be a field extension and let $x \in K^{\MW}_n(F)$. Then we have $\varphi(x) - \varphi(0) \in F_{d+n}M_*(F)$ for all non-negative integers $d$, so that $\varphi(x) = \varphi(0)$ by the fact that the intersection $\bigcap_{d \geq 0} F_{d+n}M_*(F)$ is trivial. Since elements of $M_*(k)$ clearly are in the kernel of $\partial^\pm$, we are done.
\end{proof}
\section{Computing the Operations}\label{Section 8}
Using the previously defined shifts we finally start with computing operations. We will first deal with quotients with respect to the filtration  and then lift these computations using the separatedness of the filtration. As in the last section, we let $n$ be a positive integer.
\begin{prop}\label{"locally" gen by div powers}
For all $\mathbb{N}$-graded homotopy algebras $M_*$ and all non-negative integers $d$, the $M_*(k)/F_dM_*(k)$-module $\Op^n_{\sp}(M_*)/\!\Op^n_{\sp}(F_dM_*)$ is generated by resi\-due classes of the $\lambda_i \cdot a$ for $ni < d$ with $a \in \prescript{}{\delta_nh} M_*(k)$ if $i \geq 2$.
\end{prop}
\begin{proof}
We give a proof by induction on $d \geq 0$. For $d = 0$ we have $F_dM_* = M_*$, so that the quotient $\Op^n_{\sp}(M_*)/\!\Op^n_{\sp}(F_0M_*)$ is the trivial module over the zero ring. This is certainly generated by the empty set.
\\Suppose that the statement is true for non-negative integers up to some integer $d$ and let $\varphi \in \Op^n_{\sp}(M_*)$. We denote the image of $\varphi$ under the quotient map $\Op^n_{\sp}(M_*) \rightarrow \Op^n_{\sp}(M_*)/\!\Op^n_{\sp}(F_{d+1}M_*)$ by $\overline{\varphi}$. Its positive shift $\overline{\varphi}^{(+)}$ lies in $\Op^n_{\sp}(M_*)/\!\Op^n_{\sp}(F_{d+1-n}M_*)$ and can hence by the induction hypothesis be written as $\overline{\varphi}^{(+)} = \sum_{0 \leq i \leq \frac{d+1-n}{n}}  \overline{\lambda_i \cdot a_i}$ for some $a_i \in M_*(k)$, which for $i \geq 2$ lie in $\prescript{}{\delta_nh} M_*(k)$. We now consider the operation 
$$\psi = \varphi - \sum_{0 \leq i \leq \frac{d+1-n}{n}} \lambda_{i+1} \cdot a_i \in \Op^n_{\sp}(M_*,)$$
 which is well-defined since by Lemma \ref{Properties of derivatives} we have $a_1 \in \prescript{}{\delta_nh} M_*(k)$. Indeed,
$$\overline{a_1} = ((\overline{\varphi})^{(+)})^{(+)}- \Big(\Big(\sum_{2 \leq i \leq \frac{d+1-n}{n}} \lambda_{i+1} \cdot a_i \Big)^{(+)} \Big)^{(+)},$$
so that $a_1$ is the difference of two elements in the kernel of $\prescript{}{\delta_nh} M_*(k)$  and thus itself lives there. By the definition of $\psi$ and Proposition \ref{Shift is derivative} we have $\overline{\psi}^{(+)} = 0$, which yields $\overline{\psi} = \overline{a_{-1}}$ for some element $a_{-1} \in M_*(k)$ according to Proposition \ref{ker of shifts}. Here $\overline{a_{-1}}$ denotes the residue class of $a_{-1}$ modulo $F_dM_*(k)$ considered as a constant operation. Thus $\overline{\varphi} = \sum_{-1 \leq i \leq \frac{d+1-n}{n}}\overline{\lambda_{i+1} \cdot a_i}$ as wanted.
\end{proof}
Since higher divided power operations are not necessarily well-defined by themselves but make use of elements of $\prescript{}{\delta_nh} M_*(k)$, we will need that higher shifts come with such elements. 
\begin{cor}\label{Derivatives invariant of order}
For all $\mathbb{N}$-graded homotopy algebras $M_*$ and all non-negative integers $d$, we have 
\begin{enumerate}
\item[(i)] $(\overline{\varphi}^{(+)})^{(-)} = (\overline{\varphi}^{(-)})^{(+)}$ for all $\overline{\varphi} \in\Op^n_{\sp}(M_*)/\!\Op^n_{\sp}(F_dM_*)$ and all odd $n$. In particular, $(\varphi^{(+)})^{(-)} = (\varphi^{(-)})^{(+)}$ holds for all operations $\varphi \in\Op^n_{\sp}(M_*)$ independent of the parity of $n$;
\item[(ii)] $\varphi^{(\pm)} \in \Op_{\sp}(K_n^{\MW},\prescript{}{\delta_nh} M_*(k))$ for all $\varphi \in \Op_{\sp}(K_n^{\MW},\prescript{}{\delta_nh} M_*(k))$;
\end{enumerate}
\begin{proof}
In light of Lemma \ref{Properties of derivatives}, both statements are only of interest to us in the case that $n$ is odd. The first part of (i) follows directly from the previous statement together with Proposition \ref{Shift is derivative}. Now for the second part, let  $\varphi \in\Op^n_{\sp}(M_*)$. Using the first part, the difference $(\varphi^{(+)})^{(-)} - (\varphi^{(-)})^{(+)}$ defines an element of $\Op^n_{\sp}(F_dM_*)$ for every non-negative integer $d$ and hence lies in the intersection $\bigcap_{d \geq 0}\Op^n_{\sp}(F_dM_*) = 0$.

For (ii) we also only need to consider odd $n$. Let $\varphi \in \Op_{\sp}(K_n^{\MW},\prescript{}{h} M_*)$, let $k \subset F$ be a field extension, let $x \in K^{\MW}_n(F)$ and let $\underline{a} \in (F^\times)^n$. The operation $\varphi^{(\pm)}$ is defined via the equation 
$$\varphi(x \pm [\underline{a}]) = \varphi(x) \pm [\underline{a}]\varphi^{(\pm)}(x),$$
which gives us
$$\pm [\underline{a}]\varphi^{(\pm)}(x) = \varphi(x \pm [\underline{a}]) - \varphi(x) \in \prescript{}{h} M_*(F).$$
Hence we have
$$[\underline{a}](\pm h)\varphi^{(\pm)}(x) = h(\pm [\underline{a}]) \varphi^{(\pm)}(x) = 0,$$ 
which as seen so often yields $h\varphi^{(\pm)}(x) = 0$ by Lemma \ref{[t]x = 0}. In other words we have $\varphi^{(\pm)} \in \Op_{\sp}(K_n^{\MW},\prescript{}{h} M_*)$.
\end{proof}
\end{cor}
In particular, we may apply the two shifts independently of their order and can define $\varphi^{(+m,-n)}$ as the operation $\varphi$ shifted $m$ times with respect to $\partial^{+}$ and $n$ times with respect to $\partial^{-}$. 

As mentioned before, the operations $\lambda^n_l$ turn out to essentially generate all operations. To be able to make the word ``essentially" precise, we introduce the following operations:
$$\sigma_l^n = \sum^l_{j = \lfloor \frac{l}{2} \rfloor +1} \binom{\lfloor \frac{l-1}{2} \rfloor }{j-\lfloor \frac{l}{2} \rfloor -1}[-1]^{n(l-j)}\lambda_j^n = \sum^{\lfloor \frac{l-1}{2} \rfloor}_{j = 0} \binom{\lfloor \frac{l-1}{2} \rfloor }{j}[-1]^{n(l-j)}\lambda_{l-j}^n$$
for all integers $l \geq 1$ and we additionally set $\sigma_0^n = \lambda_0^n$.
Once again we need to know the shifts of these operations:
\begin{prop}\label{Derivatives of new Lambdas}
Let $M_*$ be an $\mathbb{N}$-graded homotopy algebras and let $y \in \prescript{}{\delta_nh} M_*(k).$ Then we have
\begin{enumerate}
\item[(i)] ${(\sigma_0^n \cdot 1)}^{(\pm)} = 0$ and ${(\sigma_1^n \cdot 1)}^{(\pm)} = \sigma_0^n \cdot 1$;
\item[(ii)] ${(\sigma^n_l \cdot y)}^{(+)} = \sigma_{l-1}^n \cdot y$ and ${(\sigma^n_l \cdot y)}^{(-)} = (\sigma_{l-1}^n + [-1]^n\sigma_{l-2}^n)  \cdot y$ for $l \geq 2$ even;
\item[(iii)] ${(\sigma^n_l \cdot y)}^{(+)}= (\sigma_{l-1}^n + [-1]^n\sigma_{l-2}^n) \cdot y$ and ${(\sigma^n_l \cdot y)}^{(-)} = \sigma_{l-1}^n \cdot y$ for $l \geq 2$ odd;
\end{enumerate}
\end{prop}
\begin{proof}
This is just a computation using Proposition \ref{Shift is derivative}, which in particular already contains part (i). Let us therefore focus on (ii) and (iii). Let $l \geq 2$ be even and write $l = 2d$. Then the operation $\sigma_l \cdot y$ is given by 
$$\sigma_l^n \cdot y = \sum^{2d}_{j = d + 1} \binom{d-1}{j-d-1}[-1]^{2d-j}\lambda_j^n \cdot y,$$ 
which gives
$${(\sigma^n_l \cdot y)}^{(+)} = \sum^{2d}_{j = d + 1} \binom{d-1}{j-d-1}[-1]^{2d-j}\lambda_{j-1}^n \cdot y.$$
If $l \geq 2$ is odd, we write it as $l = 2d+1$ and get 
$$\sigma_l^n \cdot y = \sum^{2d+1}_{j = d + 1} \binom{d}{j-d-1}[-1]^{2d + 1 -j}\lambda_j^n \cdot y,$$ 
which results in
$${(\sigma^n_l \cdot y)}^{(+)} = \sum^{2d+1}_{j = d + 1} \binom{d}{j-d-1}[-1]^{2d + 1 -j}\lambda_{j-1}^n \cdot y.$$
For $l = 2d \geq 2$ we directly get
$${(\sigma^n_l \cdot y)}^{(+)} = \sum^{2d}_{j = d + 1} \binom{d-1}{j-d-1}[-1]^{2d-j}\lambda_{j-1}^n \cdot y = \sum^{2d-1}_{j = d} \binom{d-1}{j-d}[-1]^{2d-j-1}\lambda_j^n \cdot y,$$
which is exactly $\sigma_{l-1}^n \cdot y$ as written above. If $l = 2d+1 \geq 2$ is odd, we need to compare 
$${(\sigma^n_l \cdot y)}^{(+)} = \sum^{2d+1}_{j = d + 1} \binom{d}{j-d-1}[-1]^{2d + 1 -j}\lambda_{j-1}^n \cdot y = \sum^{2d}_{j = d} \binom{d}{j-d}[-1]^{2d-j}\lambda_j^n \cdot y$$
with 
$$\sum^{2d}_{j = d + 1} \binom{d-1}{j-d-1}[-1]^{2d-j}\lambda_j^n \cdot y + [-1]\sum^{2d-1}_{j = d} \binom{d-1}{j-d}[-1]^{2d-j-1}\lambda_j^n \cdot y.$$
Now these two terms agree by the standard recurrence relation for binomial coefficients. The two formulas  for the negative shifts can be shown similarly.
\end{proof}
We will form infinite sums of our operations and hence need to know that this is well-defined. The key is that these sums become finite whenever evaluated:
\begin{prop}
Let $M_*$ be an $\mathbb{N}$-graded homotopy algebra. For all elements $y$ of $\prescript{}{\delta_nh} M_*(k)$, all field extensions $k \subset F$ and all elements $x \in K_n^{\MW}(F)$, we have $\sigma_l^n \cdot y(x) = 0$ for all but finitely many $l \geq 0$.
\end{prop}
\begin{proof}
Let $M_*$ be an $\mathbb{N}$-graded homotopy algebra, let $y \in \prescript{}{\delta_nh} M_*(k)$, let $k \subset F$ be a field extension and let $x \in K_n^{\MW}(F)$. Then $x$ can be written as 
$$x = [\underline{a_1}] + \dotsc + [\underline{a_r}] - [\underline{b_1}] - \dotsc -[\underline{b_s}]$$
 for some elements $\underline{a_1},\dotsc,\underline{a_r},\underline{b_1},\dotsc,\underline{b_s} \in (F^\times)^n$ and some non-negative integers $r$ and $s$. We claim that $\sigma_l^n \cdot y(x) = 0$ for all $l \geq 2\max(r,s)+1$.
\\First note that we may assume $r = s$ by adding or subtracting $[1]^n = 0$ enough times. We now prove the claim by induction on $r \geq 0$. If $r = 0$, we have $x = 0$ and hence clearly $\sigma_l^n \cdot y(x) = 0$ for all $l \geq 1$. 

Let us now assume that the claim is true for all non-negative integers up to some $r-1$. We consider an element of the form 
$$x = [\underline{a_1}] + \dotsc + [\underline{a_r}] - [\underline{b_1}] - \dotsc -[\underline{b_r}]$$
which we will also write as $x = x' + [\underline{a_r}] - [\underline{b_r}]$. Using Proposition \ref{Derivatives of new Lambdas} we now get that 
$$\sigma_l^n \cdot y(x) = \sigma_l^n \cdot y(x') + [\underline{a_r}]{\sigma_l^n \cdot y}^{(+)}(x') - [\underline{b_r}]{\sigma_l^n \cdot y}^{(-)}(x') + [\underline{a_r},\underline{b_r}]{\sigma_l^n \cdot y}^{(+,-)}(x')$$
is some combination of elements of the form $\sigma_d^n \cdot y(x')$ with $d \geq l-2$ and therefore vanishes if $l \geq 2r+1$ by the induction hypothesis.
\end{proof}
We define a filtration on $M_*(k)^2 \times \prescript{}{\delta_n h} M_*(k)^{\mathbb{N} \setminus \lbrace 0,1 \rbrace}$ via taking
$$F_d(M_*(k)^2 \times \prescript{}{\delta_nh} M_*(k)^{\mathbb{N} \setminus \lbrace 1,2 \rbrace}) = \lbrace (a_l)_{l \geq 0} \mid a_l \in F_{\max(d-nl,0)}M_*(k) \text{ for all } l \geq 0 \rbrace$$
to be the $d$-th piece of the filtration. This allows us to present our second main result.
\begin{thm}\label{Operations K_n^MW -> M_*}
For all $\mathbb{N}$-graded homotopy algebras $M_*$ and all positive integers $n$, the two maps
$$f \colon M_*(k)^2 \times \prescript{}{\delta_n h} M_*(k)^{\mathbb{N} \setminus \lbrace 0,1 \rbrace} \rightarrow \Op_{\sp}(K^{\MW}_n, M_*), \, (a_l)_{l \geq 0} \mapsto \sum_{l \geq 0} \sigma_l^n \cdot a_l$$
and
$$g \colon \Op_{\sp}(K^{\MW}_n, M_*) \rightarrow M_*(k)^2 \times \prescript{}{\delta_n h} M_*(k)^{\mathbb{N} \setminus \lbrace 0,1 \rbrace}, \, \varphi \mapsto (\varphi^{(+\lfloor \frac{l+1}{2} \rfloor,- \lfloor \frac{l}{2} \rfloor)}(0))_{l \geq 0}$$
are mutually inverse isomorphisms of filtered $M_*(k)$-modules.
\end{thm}
\begin{proof}
First note that $f$ is well-defined by the previous Proposition. Furthermore, these two maps are clearly morphisms of $M_*(k)$-modules which preserve the respective filtrations since $\sigma_l^n$ takes values in $F_lM_*$ for non-negative integers $l$ and each application of $\partial^\pm$ shifts the filtration by $n$ as seen in Proposition \ref{Shifts}. 

Next we show that $f$ is a right inverse of $g$. Let $(a_l)_{l \geq 0 } \in M_*(k)^2 \times \prescript{}{\delta_nh} M_*(k)^{\mathbb{N} \setminus \lbrace 1,2 \rbrace}$. Note that by Proposition \ref{ker of shifts}, we can pretend that $f((a_l)_{l \geq 0})$ is a finite sum to compute its image under the map $g$. If $d$ is even, we therefore have
$$(f((a_l)_{l \geq 0}))^{(+\lfloor \frac{d+1}{2} \rfloor,- \lfloor \frac{d}{2} \rfloor)} = \left(\sum_{l \geq 0} \sigma_l^n \cdot a_l \right)^{\hspace*{-4pt}(+\lfloor \frac{d+1}{2} \rfloor,- \lfloor \frac{d}{2} \rfloor)} = \sum_{l \geq 0}  \sigma_l^n \cdot a_{d+l}$$ 
according to Proposition \ref{Derivatives of new Lambdas}. If $d$ is odd, we simply need to compute the positive shift of this operation, which by the same Proposition is
$$\sigma_0^n \cdot a_d + \sigma_1^n \cdot a_{d+1} + (\sigma_2^n + [-1]\sigma_1^n) \cdot a_{d+2} + \sigma_3^n \cdot a_{d+3} + (\sigma_4^n + [-1]\sigma_3^n) \cdot a_{d+4} + \dotsc$$
Plugging in $0$ in both cases hence gives $g(f((a_l)_{l \geq 0})) = (a_l)_{l \geq 0}$ as wanted.

Finally we show that the kernel of $g$ is trivial. Let $\varphi \in \ker(g)$, in other words we have $\varphi^{(+\lfloor \frac{l+1}{2} \rfloor,- \lfloor \frac{l}{2} \rfloor)}(0) = 0$ for all non-negative integers $l$. By Proposition \ref{"locally" gen by div powers} and the definition of the operations $\sigma^n_l$, we have $\overline{\varphi} = \sum_{i = -1}^{d-1}  \overline{\sigma_{i+1}^n \cdot a_i}$ for some $\overline{a_0},\dotsc,\overline{a_{d-1}} \in M_*(k)/F_dM_*(k)$, where we consider $\varphi$ modulo $\Op_{\sp}(K^{\MW}_n,F_dM_*)$. Therefore we get $$a_l = \varphi^{(+\lfloor \frac{l+1}{2} \rfloor,- \lfloor \frac{l}{2} \rfloor)}(0) = 0 \text{ modulo } F_dM_*(k)$$ for all $0 \leq l \leq d-1$. Thus all the $a_i$ live in $F_dM_*(k)$. Since this is true for all non-negative integers $d$ and the filtration $(F_dM_*(k))_{d \geq 0}$ is separated, we have $\varphi = 0$.
\end{proof}
\begin{cor}
For every integer $m$, the $K^{\MM}_*(k)$-module $\Op_{\sp}(K^{\MW}_n, K^{\MM}_{ \geq m})$ is given by $$\prod_{l = 0}^1 \sigma_l^n \cdot K^{\MM}_{\geq m-nl}(k) \times \prod_{l \geq 2}\sigma_l^n \cdot \prescript{}{\delta_n2}K^{\MM}_{\geq m-nl}(k).$$ In particular we have that the abelian group $\Op_{\sp}(K^{\MW}_n, K^{\MM}_m)$ is given by $$\bigoplus_{\min(\frac{m}{n},1) \geq l \geq 0} \sigma_l^n \cdot K^{\MM}_{m-nl}(k) \times \hspace*{-5pt} \bigoplus_{\frac{m}{n} \geq l \geq 2}\hspace*{-6.5pt} \sigma_{l}^n\cdot \prescript{}{\delta_n2}K^{\MM}_{m-nl}(k)$$
\end{cor}
\begin{cor}
For every integer $m$, the $K^{\MM}_*(k) / 2$-module $\Op_{\sp}(K^{\MW}_n, K^{\MM}_{ \geq m} / 2)$ is given by $$\prod_{l \geq 0} \sigma_l^n \cdot K^{\MM}_{\geq m-nl}(k) / 2.$$ In particular the abelian group $\Op_{\sp}(K^{\MW}_n, K^{\MM}_m / 2)$ is given by $$\bigoplus_{\frac{m}{n} \geq l \geq 0}\hspace*{-6.5pt} \sigma_l^n \cdot K^{\MM}_{m-nl}(k) / 2.$$
\end{cor}
\begin{cor}
For every integer $n$, the $W(k)$-module $\Op_{\sp}(K^{\MW}_n, I^m)$ is given by $$\prod_{l \geq 0}\sigma_l^n \cdot I^{m-nl}(k).$$
\end{cor}
Based on these corollaries and the pullback square of Milnor-Witt K-theory, we also get the operations on Milnor-Witt K-theory.
\begin{cor}\label{Operations K^MW -> K^MW}
For all integers $m$, the abelian group $\Op_{\sp}(K^{\MW}_n, K^{\MW}_m)$ is given by 
$$\prod_{l = 0}^1 \sigma_l^n \cdot K^{\MW}_{m-nl}(k) \times \prod_{l \geq 2} \sigma_l^n \cdot \prescript{}{\delta_nh} M_{m-nl}(k).$$
\end{cor}
\begin{proof}
The pullback diagram 
\begin{center}
\begin{tikzpicture}
  \matrix (m) [matrix of math nodes,row sep=3em,column sep=3em,minimum width=2em]
  {
     K^{\MW}_m & K^{\W}_m \\
    K^{\MM}_m & K^{\MM}_m \! /2\\};
  \path[-stealth]
    (m-1-1) edge[->] (m-1-2)
    (m-1-1) edge[->] (m-2-1)
    (m-1-2) edge[->] (m-2-2)
    (m-2-1) edge[->] (m-2-2);
\end{tikzpicture}
\end{center}
gives rise to the pullback diagram
\begin{center}
\begin{tikzpicture}
  \matrix (m) [matrix of math nodes,row sep=3em,column sep=3em,minimum width=2em]
  {
     \Op_{\sp}(K^{\MW}_n, K^{\MW}_m) & \Op_{\sp}(K^{\MW}_n, K^{\W}_m) \\
    \Op_{\sp}(K^{\MW}_n, K^{\MM}_m) & \Op_{\sp}(K^{\MW}_n, K^{\MM}_m/2)\\};
  \path[-stealth]
    (m-1-1) edge[->] (m-1-2)
    (m-1-1) edge[->] (m-2-1)
    (m-1-2) edge[->] (m-2-2)
    (m-2-1) edge[->] (m-2-2);
\end{tikzpicture}
\end{center}
of operations. By the previous Corollaries, it therefore suffices to show that
$$\prod_{l = 0}^1 \sigma_l^n \cdot K^{\MW}_{m-nl}(k) \times \prod_{l \geq 2} \sigma_l^n \cdot \prescript{}{\delta_nh} K^{\MW}_{m-nl}(k)$$
is the pullback of the diagram
\begin{center}
\begin{tikzpicture}
  \matrix (m) [matrix of math nodes,row sep=3em,column sep=1em,minimum width=2em]
  {
      &  \prod_{l \geq 0}\sigma_l^n \cdot \K^{\W}_{n-ml}(k) \\  \prod_{l = 0}^1 \sigma_l^n \cdot K^{\MM}_{m-nl}(k) \times \prod_{\frac{m}{n} \geq l \geq 2}\sigma_{l}^n \cdot \prescript{}{\delta_n2}K^{\MM}_{m-nl}(k) & \prod_{\frac{m}{n} \geq l \geq 0} \sigma_{l}^n \cdot K^{\MM}_{m-nl}(k) / 2\\};
  \path[-stealth]
    (m-1-2) edge[->] (m-2-2)
    (m-2-1) edge[->] (m-2-2);
\end{tikzpicture}
\end{center}
which is clear by the fact that pullbacks and products commute.
\end{proof}
All of this also gives us a lot of operations on $K_n^{\MW}$ for non-positive $n$. Recall from Section \ref{Section 5} that there are operations 
$$y \cdot (\langle - \rangle \mapsto [-]) \colon K_0^{\MW} \rightarrow K_m^{\MW}$$ 
for $y \in \prescript{}{h} K^{\MW}_{m-1}(k)$. Due to this condition on $y$, also $y \cdot (\eta^n\langle - \rangle \mapsto [-])$ is a well-defined operation $K_n^{\MW} \rightarrow K_m^{\MW}$, where $n$ is non-positive and we can compose it with any of our operations on $K_m^{\MW}$, which for positive $m$ gives us a substantial amount of operations.
\section{Recovering Garrel's and Vial's Operations}\label{Section 9}
Now that we understand operations on Milnor-Witt K-theory, let us reprove the known results on operations on Milnor K-theory by Vial \cite{MR2497581} and Witt K-theory by Garrel \cite{MR4113769}. Let $r$ be a positive integer. Given some further integers $s_{i_d}$ indexed by a subset $\lbrace i_1,\dotsc,i_j \rbrace \subset \lbrace 1,\dotsc,r \rbrace$, we denote by $e_{(s_{i_d})}$ the number of even and by $o_{(s_{i_d})}$ the number of odd integers among $(s_{i_d}) = (s_{i_1},\dotsc,s_{i_j})$.
\begin{lem}\label{Operations when adding hK^MW_*}
Let $n$ be a positive integer, let $M_*$ be an $\mathbb{N}$-graded homotopy algebra and let $\varphi \in \Op_{\sp}(K^{\MW}_n,M_*)$. Then we have
$$\varphi \Bigl( x + h\sum_{i = 1}^r (-1)^{s_i}[\underline{a_i}] \Bigr) \hspace*{-2pt}= \hspace*{-1pt}\varphi(x) + \sum_{j = 1 }^r h^j \hspace*{-10pt}\sum_{1 \leq i_1 < \dotsc < i_j \leq r}\hspace*{-18pt}(-1)^{\sum_{d = 1}^j \hspace*{-1pt}s_{i_d}} \hspace*{-3pt}\prod_{d = 1}^j [\underline{a_{i_d}}]\varphi^{(+e_{(s_{i_d})},-o_{(s_{i_d})})}(x)$$
for all $x \in K^{\MW}_n(F)$, $\underline{a_1},\dotsc,\underline{a_r} \in (F^\times)^n$, all field extensions $k \subset F$ and all positive integers $r$ and $s_1,\dotsc,s_r$.
\end{lem}
\begin{proof}
We give a proof by induction on $r \geq 1$. Let $k \subset F$ be a field extension and let $x \in K^{\MW}_n(F)$. If $a_1,\dotsc,a_n \in F^\times$, then 
\begin{align*}
\varphi(x \pm h[a_1,\dotsc,a_n]) = \varphi(x \pm [a_1^2,a_2,\dotsc,a_n]) &= \varphi(x) \pm [a_1^2,a_2,\dotsc,a_n]\varphi^{(\pm)}(x) \\ &= \varphi(x) \pm h[a_1,\dotsc,a_n]\varphi^{(\pm)}(x),
\end{align*}
which clarifies the $r = 1$ case. Now suppose the statement is true for some positive integer $r$ and let $\underline{a_1},\dotsc,\underline{a_{r+1}} \in (F^\times)^n$. Then we have 
$$\varphi \Bigl( x + h\sum_{i = 1}^{r+1} (-1)^{s_i}[\underline{a_i}] \Bigr) = \varphi \Bigl( x + h\sum_{i = 1}^r (-1)^{s_i}[\underline{a_i}] \Bigr) \pm h[\underline{a_{r+1}}]\varphi^{(\pm)}\Bigl( x + h\sum_{i = 1}^r (-1)^{s_i}[\underline{a_i}] \Bigr).$$
Using the induction hypothesis for both summands and regrouping everything clearly yields the claimed formula.
\end{proof}
We denote by $\overline{\sigma}^n_l \cdot y$ the operations on the quotient $K^{\MW}_n /  hK^{\MM}_n$ induced by $\sigma^n_l \cdot y$, if they are well-defined. Since the isomorphism $K^{\MW}_n /  hK^{\MW}_n \rightarrow K^{W}_n (\cong I^n)$ maps $[\underline{a}] + hK^{\MM}_n$ to $-\lbrace \underline{a} \rbrace$ (or further to $-\langle \langle \underline{a} \rangle \rangle$), 
the operation $\overline{\sigma}^n_l \cdot y$ corresponds to the operation $g^l_n \cdot y$ of Garrel, but does not coincide with it under the above isomorphism due to the change of sign. The operations $g^l_n \cdot y$ are defined via certain operations $f^d_n \cdot y$, which are the ones corresponding to our operations of the form $\lambda^n_d \cdot y$. We can also define the operations $f^d_n \cdot y$ on the level of Milnor-Witt K-theory, by mapping $[\underline{a}]$ to $(1+[\underline{a}]t)^{-1} \cdot y$ instead of $(1+[\underline{a}]t) \cdot y$ and then repeating the proof of Proposition \ref{Definition of Lambda}. Then one obtains the relation
$$f^l_n \cdot y = (-1)^l \sum_{i = 0}^{l-1}\binom{l-1}{i}[-1]^{ni}\lambda_{l-i}^n \cdot y$$
for all positive integers $l$ and $n$ by a simple induction. The same formula also holds if one starts with $\lambda^n_l \cdot y$ and wishes to express it via operations of the form $f^d_n \cdot y$. This allows us to go back and forth between our operations and the ones of Garrel.
\begin{prop}
For all positive integers $n$ and all $\mathbb{N}$-graded homotopy algebras $M_*$, we have 
$$\Op_{\sp}(I^n,M_*) \cong \Bigl\{ \sum_{l \geq 0} \overline{\sigma}_l^n \cdot a_l \mid (a_l)_{l \geq 0} \in M_*(k) \times \prescript{}{h} M_*(k)^{\mathbb{N} \setminus \lbrace 0 \rbrace} \Bigr\}$$
as a filtered $M_*(k)$-module. In particular we recover Theorem 4.9 of \cite{MR4113769}, if $M_* = W$ or $M_* = K^{\MM}_* / 2 \cong H^*(-,\mu_2)$. 
\end{prop}
\begin{proof}
We need to determine those operations $\varphi \in \Op_{\sp}(K^{\MW}_n,M_*)$ satisfying 
$$\varphi(x + hK_n^{\MW}(F)) = \varphi(x)$$
for all $x \in K_n^{\MW}(F)$ and all field extensions $k \subset F$. By the previous Lemma and Lemma \ref{[t]x = 0}, operations of the form $\sum_{l \geq 0} \sigma_l^n \cdot a_l$ with $(a_l)_{l \geq 0} \in M_*(k) \times \prescript{}{h} M_*(k)^{\mathbb{N} \setminus \lbrace 0 \rbrace}$ do exactly that. Therefore it remains to show that these are the only such operations. Let $\varphi \in \Op_{\sp}(K^{\MW}_n,M_*)$ with $\varphi(x + hK_n^{\MW}(F)) = \varphi(x)$ for all $x \in K_n^{\MW}(F)$ and all field extensions $k \subset F$. Picking $\pm h[b_1,\dotsc,b_n] \in hK_n^{\MW}(F)$, Lemma \ref{Operations when adding hK^MW_*} tells us that $\pm h[b_1,\dotsc,b_n]\varphi^{(\pm)}(x) = 0$. Thus we have $h\varphi^{(\pm)}(x) = 0$ due to Lemma \ref{[t]x = 0}, which in light of Theorem \ref{Operations K_n^MW -> M_*} and Proposition \ref{Derivatives of new Lambdas} means that $\varphi = \sum_{l \geq 0} \sigma_l^n \cdot a_l$ with sequence of coefficients from $M_*(k) \times \prescript{}{h} M_*(k)^{\mathbb{N} \setminus \lbrace 0 \rbrace}$ as claimed. 
\end{proof}
Let us now deal with Vial's operations. For this we will first derive a formula which explains what happends if we add elements of the form $\pm \eta[a,b,\underline{c}]$ before applying an operation, at least when $\eta$ acts trivially on $M_*$. Note that such $M_*$ are equivalent to Rost's notion of cycle modules \cite{MR1418952} with ring structure, see Remark 2.50 of \cite{MR2934577} or Section 12 of \cite{MR4282798} together with Theorem 4.0.1 of \cite{MR4282798}.
\begin{lem}\label{Operations when adding etaK^MW_*}
Let $n$ be a positive integer, let $M_*$ be a cycle module with ring structure and let $\varphi \in \Op_{\sp}(K^{\MW}_n,M_*)$. Then we have
$$\varphi( x \pm \eta[a,b,c_1,\dotsc,c_{n-1}]) = \varphi(x) - [a,b,c_1,\dotsc,c_{n-1}][-1]^{n-1}\varphi^{(\mp2)}(x)$$
for all $x \in K^{\MW}_n(F)$, $a,b,c_1,\dotsc,c_{n-1} \in F^\times$ and all field extensions $k \subset F$.
\end{lem}
\begin{proof}
Let $k \subset F$ be a field extension, let $a,b,c_1,\dotsc,c_{n-1} \in F^\times$, $x \in K^{\MW}_n(F)$ and let $\varphi \in \Op_{\sp}(K^{\MW}_n,M_*)$. We set $\underline{c} = (c_1,\dotsc,c_{n-1})$. Then we have 
$$\eta[a,b,\underline{c}] = [ab,\underline{c}] - [a,\underline{c}] - [b,\underline{c}]$$
which gives us 
\begin{align*}
\varphi(x \pm \eta[a,b,\underline{c}]) & = \varphi(x) \pm [ab,\underline{c}]\varphi^{(\pm)}(x) \mp [a,\underline{c}]\varphi^{(\mp)}(x) \mp [b,\underline{c}]\varphi^{(\mp)}(x) \\ & - [ab,\underline{c},a,\underline{c}]\varphi^{(+,-)}(x) - [ab,\underline{c},b,\underline{c}]\varphi^{(+,-)}(x) + [a,\underline{c},b,\underline{c}]\varphi^{(\mp2)}(x) \\ & \pm [ab,\underline{c},a,\underline{c},b,\underline{c}]\varphi^{(\pm 1, \mp 2)}(x)
\end{align*}
by applying Proposition \ref{Shifts} various times. The second line is given by
$$-(2[a,b,\underline{c}][-1]^{n-1} + [-1]^n([a,\underline{c}] + [b,\underline{c}])\varphi^{(+,-)}(x) + [a,b,\underline{c}][-1]^{n-1}\varphi^{(\mp2)}(x)$$
and using Proposition \ref{Properties of derivatives} (iii), we can replace the first line by 
$$\varphi(x) + [-1]^n([a,\underline{c}] + [b,\underline{c}])\varphi^{(+,-)}(x).$$
Therefore we have
\begin{align*}
\varphi(x \pm \eta[a,b,\underline{c}]) & = \varphi(x) -2[a,b,\underline{c}][-1]^{n-1}\varphi^{(+,-)}(x)  + [a,b,\underline{c}][-1]^{n-1}\varphi^{(\mp2)}(x) \\ & \pm 2[a,b,\underline{c}][-1]^{n-1}[-1]^{n}\varphi^{(\pm 1, \mp 2)}(x).
\end{align*}
Proposition \ref{Properties of derivatives} (iii) now gives us
\begin{align*}
\pm 2[a,b,\underline{c}][-1]^{n-1}[-1]^{n}\varphi^{(\pm 1, \mp 2)}(x) & = 2[a,b,\underline{c}][-1]^{n-1} \varphi^{(+,-)}(x) \\ & - 2[a,b,\underline{c}][-1]^{n-1}\varphi^{(\mp2)}(x))
\end{align*}
which yields $\varphi( x \pm \eta[a,b,c]) = \varphi(x) - [a,b,\underline{c}][-1]^{n-1}\varphi^{(\mp2)}(x)$ as claimed.
\end{proof}
As also observed by Garrel in \cite{MR4113769} with respect to the mod 2 case, Vial forgot to explicitly mention that his operations $K^{\MM}_n \rightarrow M_*$ are uniformly bounded. Here $M_*$ is a cycle module. Therefore we will be able to find more operations than are listed in \cite{MR2497581}. From now on we denote the action of $[-1]^{n-1}$ on some homotopy module $M_*$ by $\tau_n$. As for the operations on Witt K-theory, we denote by $\overline{\sigma}^n_l \cdot y$ the operations on the quotient $K^{\MM}_n = K^{\MW}_n / \eta K^{\MM}_{n+1}$ induced by $\sigma^n_l \cdot y$, if they are well-defined. 
\begin{prop}\label{Operations on Milnor K-theory}
For all positive integers $n$ and all cycle modules with ring structure $M_*$, we have 
$$\Op_{\sp}(K_n^{\MM},M_*) = \Bigl\{ \sum_{l \geq 0} \overline{\sigma}_l^n \cdot a_l \mid (a_l)_{l \geq 0} \in M_*(k)^2 \times \prescript{}{\delta_n 2}{ \Bigl( \prescript{}{\tau_n} M_*(k) \Bigr)^{\mathbb{N} \setminus \lbrace 0,1 \rbrace} }  \Bigr\}$$
as a filtered $M_*(k)$-module. In particular we recover Theorem 5.5 of \cite{MR2497581}.
\end{prop}
\begin{proof}
We need to find the operations $\varphi \in \Op_{\sp}(K^{\MW}_n,M_*)$ satisfying 
$$\varphi(x + \eta K_{n+1}^{\MW}(F)) = \varphi(x)$$
for all $x \in K_n^{\MW}(F)$ and all field extensions $k \subset F$. Since  every element of $\eta K_{n+1}^{\MW}(F)$ for a field extension $k \subset F$ can be written as a sum of elements of the form $\pm \eta[a_1,\dotsc,a_{n+1}]$, operations of the form $\sum_{l \geq 0} \sigma_l^n \cdot a_l$ with coefficients $(a_l)_{l \geq 0} \in M_*(k)^2 \times \prescript{}{\delta_n 2}{ ( \prescript{}{\tau_{n-1}} M_*(k))^{\mathbb{N} \setminus \lbrace 0,1 \rbrace} }$ do that by the previous Lemma and Lemma \ref{[t]x = 0}.

No we show that these are the only such operations. Let $\varphi \in \Op_{\sp}(K^{\MW}_n,M_*)$ with $\varphi(x + \eta K_{n+1}^{\MW}(F)) = \varphi(x)$ for all $x \in K_n^{\MW}(F)$ and all field extensions $k \subset F$. Picking $\pm \eta[a_1,\dotsc,a_{n+1}]$, Lemma \ref{Operations when adding etaK^MW_*} gives us that 
$$[a_1,\dotsc,a_{n+1}][-1]^{n-1}\varphi^{(\mp2)}(x) = 0.$$ 
Therefore $[-1]^{n-1}\varphi^{(\mp2)}(x) = 0$ due to Lemma \ref{[t]x = 0}, which by Theorem \ref{Operations K_n^MW -> M_*} and Proposition \ref{Derivatives of new Lambdas} means that $\varphi = \sum_{l \geq 0} \sigma_l^n \cdot a_l$ with coefficients $(a_l)_{l \geq 0}$ from the product $M_*(k)^2 \times \prescript{}{\delta_n 2}{ ( \prescript{}{\tau_n} M_*(k))^{\mathbb{N} \setminus \lbrace 0,1 \rbrace} }$.

For the ``in particular part'', note that it does not matter whether we consider linear combinations of the operations $\lambda^n_l$ or $\sigma^n_l$ when working with uniformly bounded operations. This follows from the second part of Proposition 4.6 of \cite{MR4113769} by substituting $\sigma^n_d$ for $g_n^d$ and $\lambda^n_d$ for $f_n^d$.
\end{proof}
We can of course also compute operations $K^{\W}_n \cong I^n \rightarrow K^{\MW}_m$ and  $K^{\MM}_n \rightarrow K^{\MW}_m$ analogously as we did for Corollary \ref{Operations K^MW -> K^MW}. Together with Proposition \ref{Morphisms between Unramified Sheaves} and the Corollaries of Theorem \ref{Operations K_n^MW -> M_*} this yields:
\begin{thm}\label{Operations between homotopy modules}
For all positive integers $n$, table 1 gives a complete list of operations of degree $(n,m)$ between Milnor, Witt and Milnor-Witt K-theory.
\end{thm}
Since algebraic K-theory agrees with Milnor K-theory in degree $1$, Theorem \ref{Operations on Milnor K-theory} also gives us all operations $\KK^{\Q}_1 \rightarrow \KK^{\Q}_*$ and $\KK^{\Q}_1 \rightarrow \KK^{\Q}_m$ for arbitrary $m$. Let us record what our results yield for higher degrees.
\begin{rem}
If $n \geq 2$, we still obtain a large set of operations $\KK^{\Q}_n \rightarrow \KK^{\Q}_*$ and $\KK^{\Q}_n \rightarrow \KK^{\Q}_m$. There is the so-called Suslin-Hurewicz map $\KK^{\Q}_n \rightarrow \KK^{\MM}_n$, which can be defined using the $\mathbb{A}^1$-fiber sequence 
$$\mathbb{A}^{n+1} \setminus \lbrace 0 \rbrace \rightarrow \BGL_n \rightarrow \BGL_{n+1},$$
coming from the canonical inclusion $\GL_n \hookrightarrow \GL_{n+1}$, see \cite{MR4050101}. Theorem \ref{Operations on Milnor K-theory} thus also yields all operations $\KK^{\Q}_n \rightarrow \KK^{\Q}_*$ factorizing over the Suslin-Hurewicz map. This does of course raise the question what the image of the Suslin-Hurewicz map is. On page 370 of \cite{MR0750690} Suslin conjectured that said image is given by $(n-1)!K^{\MM}_n(F)$ for any infinite field $F$. He also showed that the $n = 3$ case of this conjecture is equivalent to the Milnor conjecture on quadratic forms in degree $3$, thus justifying the interest in his conjecture. For more on the current state of this still widely open conjecture, we refer the reader to \cite{MR4050101}, where the authors also deal with the case $n = 5$ for fields of characterstic not $2$ or $3$. Under the same assumptions the $n = 4$ case was more recently treated by Röndigs in \cite{2103.11831}.
\end{rem}

\bibliographystyle{siam}
\bibliography{references}

\end{document}